\documentclass[letterpaper,11pt,reqno]{amsart} 
\usepackage[portrait,margin=3cm]{geometry} 
\usepackage{mathrsfs} 
\usepackage[colorlinks=true,linkcolor=blue,citecolor=blue,urlcolor=blue]{hyperref} 
\usepackage{amssymb,amsthm,amsfonts,amsbsy,latexsym,dsfont,color,graphicx, color,fancyhdr, pgfplots}
\usepackage[msc-links,numeric,initials,nobysame]{amsrefs} 

\usepackage[foot]{amsaddr}
\usepackage{enumerate, enumitem}
\setlist[enumerate,1]{label={\upshape(\roman*)}}
\numberwithin{equation}{section}
\theoremstyle{plain}
\newtheorem{theorem}{Theorem}[section]

\newtheorem{lemma}[theorem]{Lemma} 
\theoremstyle{definition}

\newtheorem{remark}[theorem]{Remark}

\newtheorem{example}[theorem]{Example}

\theoremstyle{remark}

\newcommand{\cB}{\mathcal{B}}

\newcommand{\cM}{\mathcal{M}}

\newcommand{\cP}{\mathcal{P}}

\newcommand{\cS}{\mathcal{S}}

\newcommand{\cW}{\mathcal{W}}

\newcommand{\fG}{\mathfrak{G}}

\newcommand{\N}{\mathbb{N}}

\newcommand{\R}{\mathbb{R}}
\newcommand{\C}{\mathbb{C}}

\newcommand{\bA}{\mathbb{A}}
\newcommand{\bB}{\mathbb{B}}

\newcommand{\sF}{\mathscr{F}}
\newcommand{\sG}{\mathscr{G}}

\newcommand{\sP}{\mathscr{P}}

\newcommand{\dist}{\mathrm{dist}}

\renewcommand{\Re}{\operatorname{Re}}

\newcommand{\wh}{\widehat}
\newcommand{\wt}{\widetilde}
\newcommand{\ds}{\displaystyle}
\newcommand{\varep}{\varepsilon}
\newcommand{\ol}{\overline}

\newcommand{\I}{\mathrm{I}}
\renewcommand{\H}{\mathrm{H}}

\newcommand{\Tal}{\mathrm{Tal}}

\newcommand{\BH}{\mathrm{BH}}
\newcommand{\HY}{\mathrm{HY}}

\newcommand{\norm}[1]{\| #1 \|}

\newcommand{\Paren}[1]{\left( #1 \right)}

\newcommand{\brk}[1]{\langle #1 \rangle}

\renewcommand{\d}{\mathrm{d}_{\cM}}

\newcommand{\indi}[1]{\mathds{1}_{#1}}

\begin{document}
\title[Instability results for the LSI]{Instability results for the logarithmic Sobolev inequality and its application to related inequalities}
\author{Daesung Kim}
\address{University of Illinois at Urbana Champaign, 1409 W Green Street, Urbana, Illinois 61801}
\email{daesungk@illinois.edu}
\thanks{\it The author was supported in part by NSF grant \#1403417-DMS; Rodrigo Ba\~nuelos PI}
\maketitle
\begin{abstract}
We show that there are no general stability results for the logarithmic Sobolev inequality in terms of the Wasserstein distances and $L^{p}(d\gamma)$ distance for $p>1$.  To this end, we construct a sequence of centered probability measures such that the deficit of the logarithmic Sobolev inequality converges to zero but the relative entropy and the moments do not, which leads to instability for the logarithmic Sobolev inequality. As an application, we prove instability results for Talagrand's transportation inequality and the Beckner--Hirschman inequality. 
\end{abstract}

\section{Introduction}
Let $d\gamma=(2\pi)^{-\frac{n}{2}}e^{-\frac{|x|^{2}}{2}}\,dx$ be the standard Gaussian measure on $\R^{n}$ and $d\mu=fd\gamma$ a probability measure on $\R^n$ where $f$ is a nonnegative function in $L^{1}(d\gamma)$. The Fisher information $\I(\mu)$ and the relative entropy $\H(\mu)$ of $\mu$ with respect to $\gamma$ are defined by
\begin{align*}
	\I(\mu):=\int_{\R^{n}}\frac{|\nabla f|^{2}}{f}\,d\gamma,\qquad
	\H(\mu):=\int_{\R^{n}}f\log f \,d\gamma.
\end{align*}
The classical logarithmic Sobolev inequality (henceforth referred to as the LSI) states that 
\begin{align}\label{eq:LSI}
	\delta(\mu):=\frac{1}{2}\I(\mu)-\H(\mu)\geq 0. 
\end{align}
We call $\delta(\mu)$ the deficit of the LSI. If $d\mu=fd\gamma$, then we simply write $\I(f), \H(f)$, and  $\delta(f)$. Note that the constant $\frac{1}{2}$ is dimension-free and best possible. 

The characterization of equality cases in~\eqref{eq:LSI} was proven by Carlen~\cite{Carlen1991a}. He derived a Minkowski-type inequality and the strict superadditivity for the Fisher information. Combining these with the factorization theorem, he showed that equality holds in~\eqref{eq:LSI} if and only if $f(x)=\exp(b\cdot x -\frac{1}{2}|b|^{2})$ for some $b\in\R^{n}$. Note that the Gaussian measure $\gamma$ is the only centered optimizer.

Carlen also provided an alternative proof for the characterization of equality cases based on the Beckner--Hirschman entropic uncertainty principle, which  was conjectured by Hirschman~\cite{Hirschman1957a} and proven by Beckner~\cite{Beckner1975a}. Indeed, he showed that $\delta(\mu)$ is bounded below by the relative entropy of the Fourier--Wiener transform. Then, equality cases in~\eqref{eq:LSI} follows from the fact that the relative entropy of the Fourier--Wiener transform vanishes if and only if $\mu$ is a Gaussian measure.

After equality cases were fully understood, there has been much effort to find quantitative improvement of the log Sobolev inequality. Carlen~\cite{Carlen1991a} found the lower bound of the deficit in terms of the Fourier--Wiener transform as mentioned above. Otto and Villani~\cite{Otto2000a} exploited the HWI inequality to derive the lower bound of the deficit in terms of the Fisher information and the quadratic Wasserstein distance $W_2$ (see~\eqref{eq:HWI-cons}).

In particular, there has been a great deal of interest in finding quantitative improvement of LSI in terms of functionals that quantify how far a measure is away from the optimizers.  Let $\cM$ be a family of centered probability measures $fd\gamma$ such that $\I(f)$ and $\H(f)$ are well-defined. Let $\d$ be a distance (or a functional that identifies the equality cases) in $\cM$. We say that the LSI is \emph{weakly $\d$--stable} in $\cM$ if $\delta(\mu_{k})\to 0$ and $\mu_{k}\in\cM$ imply $\d(\mu_{k}, \gamma)\to 0$. We say that the LSI is $\d$--\emph{stable} if a modulus of continuity is explicit, that is, there exists a modulus of continuity $\omega$ such that $\delta(\mu)\geq \omega(\d(\mu,\gamma))$ for all $\mu\in\cM$.

The first quantitative LSI in terms of metrics was discovered in~\cite{Indrei2014a}. Indrei and Marcon used the optimal transportation to obtain a lower bound of the deficit of the LSI in terms of the eigenvalues of the Hessian of the optimal transportation potential. Then, they applied Caffarelli's contraction theorem~\cite{Caffarelli2000a} and its generalization due to  Kolesnikov~\cite{Kolesnikov2013a}, which leads to $W_2$--stability for the LSI. We note that the potential is a solution to the Monge--Ampere equation under some regularity assumptions on the densities, and that these results of Caffarelli and Kolesnikov can be thought of as Sobolev type estimates of the equation. 

A strict improvement of the LSI for the class of probability measures satisfying a $(2,2)$-Poincar\'e inequality was proved in~\cite{Fathi2016a}, which yields stability bounds with respect to $W_{2}$ and $L^{1}(d\gamma)$. Using the scaling asymmetry of the Fisher information and the relative entropy, it was shown in~\cite{Bobkov2014a} (see also~\cite{Dolbeault2016a}*{Theorem 1} and~\cite{Bolley2018a}) that the LSI is $W_{2}$--stable in the space of probability measures whose second moments are bounded by the second moment of the standard Gaussian measure (which is the same as the dimension of the underlying space). In~\cite{Feo2017a}*{Proposition 4.7}, the authors proved $L^2$--stability (and so  $L^1$--stability) in the space of probability measures satisfying a positivity condition on the Fourier transform. Recently, Indrei and the author in~\cite{Indrei2018a} proved $W_{1}$--stability as well as $L^{1}$--stability (only in the one dimension case) in the space of probability measures with bounded second moments, where $W_{1}$ is the Kantorovich--Rubinstein distance. 
In~\cite{LNP}, the authors investigated the distance functionals induced by the Stein characterization, and proved stability results for the LSI in terms of these functionals using the Ornstein--Uhlenbeck semigroup. Recently, Gozlan~\cite{Gozlan} showed that a certain form of stability estimates of the LSI is equivalent to the Mahler conjecture, which states that the product of the volumes of a convex body and its polar body is minimized when the convex body is a hypercube.
 
Given such effort to find stability for the LSI in terms of different assumptions and distance functionals, a natural question is to determine the best possible conditions on probability measure and distances for stability for the LSI. The goal of the paper is to investigate conditions under which stability for the LSI fails. To this end, we construct sequences of probability measures such that the deficit of the LSI converges to 0 but the relative entropy does not. It turns out that our examples yield several instability results for the LSI in terms of the Wasserstein distances and $L^p(d\gamma)$ distances. The results imply that some of the existing stability estimates cannot be improved in terms of spaces of probability measures or distances. Moreover, we apply our examples to Talagrand's transportation inequality and the Beckner--Hirschman inequality to obtain instability results.

\subsection{The log Sobolev inequality}
For a probability measure $\mu$ on $\R^{n}$ and $p\geq 1$, the $p$-th moment of $\mu$ is defined by 
\begin{align*}
	m_{p}(\mu)=\int_{\R^{n}}|x|^{p}\,d\mu.
\end{align*}
The space of probability measures on $\R^{n}$ with finite $p$-th moments is denoted by $\cP_{p}(\R^{n})$. The Wasserstein distance of order $p$ between two probability measures $\mu,\nu\in \cP_{p}(\R^{n})$ is defined by
\begin{align*}
	W_{p}(\mu,\nu)=\inf_{\pi} \left(\iint |x-y|^{p}\,d\pi(x,y)\right)^{\frac{1}{p}},
\end{align*}
where the infimum is taken over all probability measures $\pi$ on $\R^{n}\times \R^{n}$ with marginals $\mu$ and $\nu$. In particular, $W_{1}$ is called the Kantorovich--Rubinstein distance and $W_{2}$ is called the quadratic Wasserstein distance. 

Let $M>0$ and $\cP_2^M(\R^n)$ be the space of probability measures $\mu$ on $\R^n$ with $m_2(\mu)\leq M$. Note that the standard Gaussian measure $d\gamma$ belongs to $\cP_2^M(\R^n)$ for $M>n$ and is the unique optimizer of the log Sobolev inequality in $\cP_2^n(\R^n)$. Note also that the space $\cP_2^M(\R^n)$ for $M>n$ has other optimizers of the form $e^{b\cdot x -|b|^2/2}d\gamma$ for some $b\in\R^n$. Note that the standard Euclidean logarithmic Sobolev inequality, which is equivalent to~\eqref{eq:LSI}, is not invariant under scaling. Optimizing in the scaling parameter, $W_2$--stability was derived in~\cite{Dolbeault2016a}*{Theorem 1} (see also~\cite{Bobkov2014a}), which states that if a probability measure $\mu$ on $\R^n$ is centered and its second moment is bounded by $n$ (that is, $\mu\in\cP_2^n(\R^n)$), then
\begin{align*}
    \delta(\mu)\geq C_n W_2^4(\mu,\gamma).
\end{align*}
A natural questions is whether the same stability holds without the moment assumption. Our first main result shows that the stability in terms of $W_2$ and $L^p$ ($p>1$) does not hold for centered probability measures whose second moments are bounded by $M$ for $M>n$. The result also implies that the $L^{1}$--stability estimate in~\cite{Indrei2018a}*{Theorem 1.1} for $n=1$ cannot be improved in terms of the $L^{p}$ distances.

\begin{theorem}\label{thm:ins-lsi-w2}
Let $M>n$ and $p>1$. There exists a sequence of centered probability measures $d\mu_{k}=f_{k}d\gamma$ in $\cP_{2}^{M}(\R^n)$ such that $\lim_{k\to\infty}\delta(\mu_{k})=0$, 
\begin{align*}
    \lim_{k\to\infty}W_{2}(\mu_{k},\gamma)=c>0,
\end{align*} 
and 
\begin{align*}
    \liminf_{k\to\infty}\|f_{k}-1\|_{L^{p}(d\gamma)}>0.
\end{align*}
\end{theorem}

Let $p>2$. By Jensen's inequality, we have $W_2(\mu,\gamma)\leq W_p(\mu,\gamma)$. Thus, it follows from Theorem~\ref{thm:ins-lsi-w2} that there is no $W_p$ stability in $\cP_2^M(\R^n)$ when $M>n$.

We note that the $L^p$-distance of probability measures can be understood as a $f$-divergence functional where $f(t)=|t-1|^p$, and the $L^2$ distance is in particular called the Pearson $\chi^2$ divergence. We also notice here that the LSI is stable in terms of the $L^p$ distance for $p > 1$ under some integrability assumptions (see~\cite{Indrei2018a}*{Corollary 1.2}).

The proofs of Theorem~\ref{thm:ins-lsi-w2} and the following results are based on the example in Lemma~\ref{lem:example}. The motivation of the proof is to consider the weighted sum of the optimizers $e^{b\cdot x - |b|^2/2}$ for the LSI. In order to facilitate to control the relevant quantities with explicit orders, we cut the overlaps of the densities of the optimizers and connect them to get a $C^\infty$ density.

Note that our example does not give an instability result for $L^1(d\gamma)$ distance. Indeed, one can see that if $d\mu_k=f_kd\gamma$ is a sequence of probability measures constructed in Lemma~\ref{lem:example}, then $\|f_k-1\|_{L^1(d\gamma)}\to 0$ as $k\to\infty$. 

\begin{remark}[Sharp exponent in $L^1$--stability]\label{rmk:lsi-sharp-L1} 
The $L^1$--stability estimate in~\cite{Indrei2018a}*{Theorem 1.1} states that if $d\mu=fd\gamma\in\cP_2^M(\R)$ is centered, then $\delta(f)\geq \|f-1\|_{L^1(d\gamma)}^4$. The higher dimensional stability estimates in terms of $L^1$ can also be found in~\cite{Indrei2018a}*{Corollary 1.4, Remark 1.5} under additional assumptions on the probability measures. It is open to determine the sharp exponent in $L^1$--stability. We note that for $\alpha<1$, the example in Lemma~\ref{lem:example} satisfies
\begin{align*}
	\lim_{k\to\infty}\frac{\delta(f_k)}{\|f_k-1\|_{L^1(d\gamma)}^\alpha}=0.
\end{align*}
Thus, it is expected that the sharp exponent will be between 1 and 4. For higher dimensions, weak $L^1$--stability in $\cP_2^M$ without any additional assumptions was proven in~\cite{Indrei2018a}*{Theorem 1.22} but the modulus of continuity is not known yet.  
\end{remark}

It was shown in~\cite{Indrei2018a} that if $\mu$ is a centered probability measure with bounded second moment (that is, $\mu\in \cP_2^M(\R^n)$), then there exists a constant $C_{n,M}>0$ such that 
\begin{align}\label{eq:W1stab-lsi}
	\delta(\mu)\geq C_{n,M}\min\{W_{1}(\mu,\gamma),W_{1}^4(\mu,\gamma)\}.
\end{align}
The next result shows that the stability in terms of $W_p$ distance for $p\ge 1$ does not holds for centered probability measures with finite second moments. As a consequence, we conclude that the stability estimate~\eqref{eq:W1stab-lsi} in terms of $W_1$ distance is sharp in terms of $\cP_2^M(\R^n)$.

\begin{theorem}\label{thm:ins-lsi-w1}
Let $p\ge 1$, then there exists a sequence of centered probability measures $d\mu_{k}=f_{k}d\gamma$ in $\cP_{2}(\R^n)$ such that $\lim_{k\to\infty}\delta(\mu_{k})=0$ and $\lim_{k\to\infty}W_{p}(\mu_{k},\gamma)=\infty$.
\end{theorem}

\begin{remark}[Sharp exponent in $W_p$--stability for $p\in[1,2)$]\label{rmk:lsi-sharp-w1}
A natural question is to find the sharp exponent in~\eqref{eq:W1stab-lsi}. Let $p\in[1,2)$, $\alpha<\frac{2p}{2-p}$, and $M>n$. 
By Lemma~\ref{lem:example} with the appropriate choice of parameters ($s=(M-n)/4$ and $t=2$, see the statement of the lemma below), one can show that there exists a  sequence of centered probability measures $\mu_k$ such that $\mu_k\in \cP_2^M(\R^n)$ for large $k$, $\delta(\mu_k)\to 0$, $W_p(\mu_k,\gamma)\to 0$, and 
\begin{align*}
	\lim_{k\to\infty}\frac{\delta(\mu_k)}{W_{p}^\alpha (\mu_k,\gamma)}=0.
\end{align*}
On the other hand, the construction of Lemma~\ref{lem:example} does not give such an example if $\alpha=2$ and $p=1$. Thus, it is expected that the sharp exponent in~\eqref{eq:W1stab-lsi} would be 2, which is an open problem. For $1<p<2$, $W_p$--stability in $\cP_2^M(\R^n)$ is not known yet. It is expected that the sharp exponent would be $\frac{2p}{2-p}$.
\end{remark}

Our instability results for the LSI allow us to compare different probability measure spaces where stability for the LSI holds. The following two remarks show that the space $\cP_2^M(\R^n)$ is different from the spaces considered in existing stability results in~\cite{Feo2017a, Indrei2018a}.

\begin{remark}
Let $\cS$ be the space of probability measures $fd\gamma$ satisfying $$\sF(e^{-\pi |x|^{2}}\sqrt{f(2\sqrt{\pi} x)})\geq 0,$$ where $\sF(\cdot)$ denotes the Fourier transform. It was shown in~\cite{Feo2017a}*{Proposition 4.7} that if $fd\gamma \in\cS$ then 
\begin{align}\label{eq:L1stab_Feo17}
	\delta(f)
	\geq \frac{1}{32}\|f-1\|_{2}^{8}.
\end{align}
We claim that $\cS\not\subset \cP_{2}^{M}(\R^{n})$ and $\cP_{2}^{M}(\R^{n})\not\subset \cS$ for any $M>0$. 
Suppose $\cP_2^M\subset \cS$. By Theorem~\ref{thm:ins-lsi-w2}, there exists a sequence of probability measures $f_k d\gamma\in\cP_2^M\subset \cS$ such that 
\begin{align*}
    \liminf_{k\to\infty}\frac{\delta(f_k)}{\|f_k-1\|_2^8}=0.
\end{align*}
In particular, one has $\delta(f_k)\le \frac{1}{64}\|f_k-1\|_2^8$ for large $k$, which contradicts to~\eqref{eq:L1stab_Feo17}. Thus, we have $\cP_{2}^{M}\not\subset \cS$ for all $M>0$. Let $f_{k}d\gamma$ be the centered Gaussian with variance $k$, then $\{f_{k}d\gamma\}$ is not included in $\cP_{2}^{M}$ for any $M>0$. Since $e^{-\pi |x|^{2}}\sqrt{f(2\pi x)}$ is also Gaussian, its Fourier transform is positive, which implies $\cS\not\subset \cP_{2}^{M}$.
\end{remark}

\begin{remark}
For $\alpha>0$ and $g\in L^{1}(d\gamma)$, we define $\cB(\alpha,g)=\{fd\gamma\in\cP: \alpha\leq f\leq g\}$.  In~\cite{Indrei2018a}*{Theorem 1.6}, the weak $L^{1}$--stability was proven in $\cB(\alpha,g)$: if $\{f_{k}d\gamma\}\subset\cB(\alpha,g)$ and $\delta(f_k)\to 0$ as $k\to\infty$ for some $\alpha>0$ and $g\in L^{1}(d\gamma)$, then $f_{k}\to 1$ in $L^{1}(d\gamma)$. For any $M,\alpha>0$ and $g\in L^{1}(d\gamma)$, we claim that $\cB(\alpha,g)\not\subset \cP_{2}^{M}(\R^{n})$ and $\cP_{2}^{M}(\R^{n})\not\subset\cB(\alpha,g)$. It suffices to consider the case $n=1$. Let $M>0$ be fixed and $f_{k}d\gamma$ be a sequence of probability measures constructed as in Lemma~\ref{lem:example} with $t=2$, and choose $s$ so that $\{f_{k}d\gamma\}\subset\cP_{2}^{M}$. Since the minimum of $f_{k}$ converges to 0, we get $\cP_{2}^{M}\not\subset\cB(\alpha,g)$. We define a sequence of functions $f_{k}$ such that $f_k(x)=f_k(-x)$ and
\begin{align*}
	f_k(x)=
	\begin{cases}
		\frac{\ds e^{\frac{x^{2}}{2}}}{\ds C_k \pi (x^{2}+1)}, & x\in[0,k],\\
		\frac{\ds e^{\frac{k^{2}}{2}}}{\ds C_k \pi (k^{2}+1)}, & x\in (k,\infty),
	\end{cases}
\end{align*}
where $C_k$ is the normalization constant so that $f_k d\gamma$ is a probability measure. Indeed one can compute $C_k$ as
\begin{align*}
	C_k = \frac{2}{\pi}\Big(\arctan(k)+\frac{\ds e^{\frac{k^{2}}{2}}(1-\Phi(k))}{k^2 +1} \Big).
\end{align*}
where $\Phi(k)=\int_\infty^k d\gamma$. Note that $C_k\to 1$ as $k\to \infty$. Furthermore, there exist $C, \alpha>0$ such that $f_{k}\geq \alpha$ for all $k$ and 
\begin{align*}
	f_{k}(x)\leq \frac{C e^{\frac{x^{2}}{2}}}{\ds  \pi (x^{2}+1)}\in L^{1}(d\gamma)
\end{align*}
for all $x$ and $k$. Since the second moment of $f_{k}d\gamma$ diverges, we conclude that $\cB(\alpha,g)\not\subset \cP_{2}^{M}(\R^{n})$.
\end{remark}

\subsection{Talagrand's transportation inequality}
Talagrand~\cite{Talagrand1996a} proved that the relative entropy is bounded below by the quadratic Wasserstein distance, that is,
\begin{align}\label{eq:Tal}
	\delta_{\Tal}(\mu):=2\H(\mu)-W_{2}^{2}(\mu,\gamma)\geq 0,
\end{align}
where $\delta_{\Tal}(\mu)$ is called the deficit of Talagrand's inequality. This inequality has a close relation to the LSI. Both the inequalities for the Gaussian measure are dimension independent, have the tensorization property, and imply the concentration phenomenon. Otto and Villani~\cite{Otto2000a} showed that a measure satisfying a log Sobolev inequality also satisfies a Talagrand-type inequality, and the converse holds under a curvature condition. From the HWI inequality
\begin{align*}
	W_{2}(\mu,\gamma)\sqrt{\I(\mu)}-\frac{1}{2}W_{2}^{2}(\mu,\gamma)
	\geq \H(\mu),
\end{align*}
one can see that the deficit of Talagrand's inequality is bounded by that of the LSI in the following sense
\begin{align}
	\delta(\mu)
    &\geq \frac{1}{2}\left(\sqrt{\I(\mu)}-W_2(\mu,\gamma)\right)^2
    \geq \frac{1}{2}\left(\sqrt{2\H(\mu)}-W_2(\mu,\gamma)\right)^2\nonumber\\
    &= \frac{\delta_{\Tal}(\mu)^2}{2\left(\sqrt{2\H(\mu)}+W_2(\mu,\gamma)\right)^2}
    \geq \frac{\delta_{\Tal}(\mu)^2}{16\H(\mu)}.\label{eq:HWI-cons}
\end{align}
In the last inequality, we used the Talagrand's transport inequality~\eqref{eq:Tal}. In particular, if $\delta(\mu_k)\to 0$ and $\H(\mu_k)\to c$ for some constant $c$, then $\delta_{\Tal}(\mu_k)\to 0$. This observation leads to the following $W_2$--instability result for Talagrand's inequality.

\begin{theorem}\label{thm:ins-ti-w2}
Let $M>n$, then there exists a sequence of centered probability measures $d\mu_{k}=f_{k}d\gamma$ in $\cP_{2}^{M}(\R^n)$ such that $\lim_{k\to\infty}\delta_{\Tal}(\mu_{k})=0$ and
\begin{align*}
    \lim_{k\to\infty}W_{2}(\mu_{k},\gamma)=c>0.
\end{align*} 
\end{theorem}

We note that an improvement of Talagrand's inequality was shown in~\cite{Mikulincer2019a}. In particular, if $\mu\in\cP_2^n(\R^n)$ then the deficit of Talagrand's inequality is bounded below by the relative entropy, which implies $W_2$--stability. It was also shown that the condition $\cP_2^n(\R^n)$ is sharp by giving an example. In one dimension, Barthe and Kolesnikov~\cite{Barthe2008a} showed that the deficit of Talagrand's inequality is bounded below by the optimal transportation cost with cost function $\varphi(z)=z-\log (1+z)$. This leads to $W_1$--stability for Talagrand's transportation inequality.  In~\cite{Fathi2016a}, the authors generalized the stability estimate to higher dimensions. In fact, they showed the $W_{1,1}$--stability bound, where $W_{1,1}$ is the $L^1$--Wasserstein distance with $\ell^1$ cost function on $\R^n$. Cordero-Erausquin~\cite{Cordero-Erausquin2017a}*{Theorem 1.3} improved the result by replacing $n^{-\frac{1}{2}}W_{1,1}$ with $W_1$. That is, it was shown that if $\mu\in \cP_2(\R^n)$, then 
\begin{align}\label{eq:Cor-Tal-W1}
	\delta_{\Tal}(\mu)\geq C\min\{W_1(\mu,\gamma),W_1^2(\mu,\gamma)\}.
\end{align} 
The next result shows that the result of~\cite{Cordero-Erausquin2017a} cannot be improved in terms of the $W_p$ distances. 

\begin{theorem}\label{thm:ins-ti-w1}
Let $p>1$, then there exists a sequence of centered probability measures $d\mu_{k}=f_{k}d\gamma$ in $\cP_{2}(\R^n)$ such that $\lim_{k\to\infty}\delta_{\Tal}(\mu_{k})=0$ and $\lim_{k\to\infty}W_{p}(\mu_{k},\gamma)=\infty$.
\end{theorem}

\begin{remark}[Sharp exponent in $W_1$--stability for Talagrand's inequality]\label{rmk:ti-sharp-w1}
Let $\alpha<1$. By Lemma~\ref{lem:example}, it is easy to see that there exists a sequence of probability measures $\mu_k$ in $\cP_2(\R^n)$ such that $\delta_\Tal(\mu_k)\to 0$, $W_1(\mu_k,\gamma)\to 0$, and
\begin{align*}
	\frac{\delta_\Tal(\mu_k)}{W_1^\alpha(\mu_k,\gamma)}\to0
\end{align*}
as $k\to\infty$. This observation implies that the exponent of $W_1$ in~\eqref{eq:Cor-Tal-W1} cannot be replaced by any smaller number than $1$. It is natural to expect that the sharp exponent would be 1. Note that if one shows~\eqref{eq:Cor-Tal-W1} with the exponent 1, then $W_1$--stability for the LSI with the sharp exponent 2 can be obtained by the proof of~\cite{Indrei2018a}, as expected in Remark~\ref{rmk:lsi-sharp-w1}. 
\end{remark}

\begin{remark}
Suppose $\mu_k$ is the sequence of probability measures constructed in Lemma~\ref{lem:example} with $t\in(0,1)$. It follows from Lemma~\ref{lem:example} and~\eqref{eq:Cor-Tal-W1} that $\delta(\mu_k)\to 0$, $W_1(\mu_k,\gamma)\to \infty$, $\delta_{\Tal}(\mu_k)\to\infty$, and
\begin{align*}
	\frac{W_1^2(\mu_k,\gamma)}{H(\mu_k)}\to 0
\end{align*}
as $k\to\infty$. This observation implies that the relative entropy term in the lower bound of~\eqref{eq:HWI-cons} is necessary.
\end{remark}

\subsection{The Beckner--Hirschman inequality}
We prove that there are no stability estimates for the Beckner--Hirschman inequality (the BHI for short) in terms of $L^{p}$ distances with specific measures and range of $p$. In this subsection, we restrict to the case $n=1$. The Shannon entropy of a nonnegative function $h$ on $\R$  with $\|h\|_{2}=1$ is given by
\begin{align*}
	S(h)=-\int_{\R}h\log h\,dx.
\end{align*}
The Beckner--Hirschman inequality states that
\begin{align*}
	\delta_{\BH}(h)	:= S(|h|^{2})+S(|\wh{h}|^{2})-(1-\log 2)\geq 0
\end{align*}
for a nonnegative function $h$ with $\|h\|_2=1$, where $\wh{h}$ is the Fourier transform defined by $\wh{h}(\xi)=\int_{\R}e^{-2\pi i x\cdot\xi}h(x)\, dx$. We call $\delta_{\BH}(h)$ the deficit of the BHI. The inequality is also called the entropic uncertainty principle. We say that a function $h$ is an optimizer for the BHI if $\delta_{\BH}(h)=0$. Let $\fG$ be the set of all nonnegative, $L^{2}$--normalized optimizers for the BHI. Using the fact that the optimizers are Gaussian (see~\cite{Lieb1990a} and~\cite{Carlen1991a}*{p.207}), we get 
\begin{align}\label{eq:BH_G_{ab}}
	\fG=\left\{G_{a,r}(x)=\Big(\frac{2a}{\pi}\Big)^{\frac{1}{4}}e^{-a(x-r)^{2}}:a>0, r\in\R\right\}.
\end{align}
We denote by $G_{a}(x):=G_{a,0}(x)$ and $g(x):=G_{\pi}(x)$. For a measure $\mu$ on $\R$ and $p>0$, we define 
\begin{align*}
	\dist_{L^{p}(d\mu)}(h,\fG)
	=\inf_{u\in\fG}\|h-u\|_{L^{p}(d\mu)}
	=\inf_{a>0, r\in\R}\|h-G_{a,r}\|_{L^{p}(d\mu)}.
\end{align*}
It was shown in~\cite{Carlen1991a} that the deficit of the LSI is bounded below by that of the BHI. To be specific, we have
\begin{align*}
	\delta_{\BH}(h)=\delta(f)-\int |\cW f|^2\log|\cW f|^2\,d\gamma \leq \delta(f),
\end{align*}
where $\cW f$ is the Fourier--Wiener transform of $f$, defined by $\cW f=\frac{1}{g}(\wh{fg})$, and 
\begin{align*}
	h(x)=(f(2\sqrt{\pi}x))^{\frac{1}{2}}g(x).	
\end{align*}
We are ready to state our instability results for the BHI.

\begin{theorem}\label{thm:ins-bhi-pw}
Let $\lambda>0$, $d\eta_{\lambda}=|x|^{\lambda}dx$, and $p\geq 2(\lambda+1)$, then there exists a sequence of nonnegative functions $\{h_{k}\}_{k\geq 1}$ in $L^{p}(d\eta_{\lambda})$ such that $\norm{h_{k}}_{2}=1$, $\delta_{\BH}(h_{k})\to 0$, $\|h_{k}\|_{L^{p}(d\eta_{\lambda})}\to \infty$, and
\begin{align*}
	\liminf_{k\to\infty}\frac{\dist_{L^{p}(d\eta_{\lambda})}(h_{k},\fG)}{\norm{h_{k}}_{L^{p}(d\eta_{\lambda})}}\geq C(p,\lambda)>0.
\end{align*}
\end{theorem}

\begin{theorem}\label{thm:ins-bhi-ew}
Let $p>\theta>0$ and $dm_{\theta}=g^{-\theta}dx$. There exists a sequence of nonnegative functions $\{h_{k}\}_{k\geq 1}$ in $L^{p}(dm_{\theta})$ such that $\norm{h_{k}}_{2}=1$, $\delta_{\BH}(h_{k})\to 0$, $\|h_{k}\|_{L^{p}(dm_{\theta})}\to \infty$, and
\begin{align*}
	\liminf_{k\to\infty}\frac{\dist_{L^{p}(dm_{\theta})}(h_{k},\fG)}{\norm{h_{k}}_{L^{p}(dm_{\theta})}}\geq C(p,\theta)>0.
\end{align*}
\end{theorem}

We emphasize that $d\eta_\lambda$ is a more suitable reference measure than $dm_{\theta}$ in a sense that $L^{p}(d\eta_{\lambda})$ contains all optimizers $\fG$ whereas $L^{p}(dm_{\theta})$ does not (see~\eqref{eq:Lpdmcondition}). If we choose the Lebesgue measure as a reference measure (that is, $\theta=0$ in Theorem~\ref{thm:ins-bhi-ew} or $\lambda=0$ in Theorem~\ref{thm:ins-bhi-pw}), then the sequence of functions $h_k$  converges to $g$ in $L^{p}$ (see Remark~\ref{rmk:bhi-lp-conv}). It remains open to show $L^{p}$--stability for the BHI with respect to the Lebesgue measure.

\subsection{Main Lemma}
The main idea of the proofs of the instability results is to consider the weighted sum of the optimizers for the LSI. Roughly speaking, we study the sum of Gaussian measures $\gamma+r(\gamma_b+\gamma_{-b})$ where $r>0$ and $\gamma_b$ is the Gaussian measure with barycenter $b$. We then observe the behaviors of the deficit of the LSI and other quantities such as the relative entropy and the Wasserstein distances when the barycenter $b$ is large and the weight $r$ is small. It turns out that the deficit of the LSI does not see the barycenter and depends only on the weight asymptotically. Since other quantities rely on both $b$ and $r$, the example leads to several types of instability results. The observation is summarized in the following lemma.

\begin{lemma}\label{lem:example}
For any $s,t>0$, there exists a sequence of centered probability measures $\mu_k$ on $\R^n$ such that 
\begin{enumerate}
	\item $\delta(\mu_k)=\frac{st}{2}k^{-t}\log k+o(k^{-t}\log k)$, 
	\item $\H(\mu_k)=sk^{2-t}-\frac{st}{2}k^{-t}\log k+o(k^{-t}\log k)$,
	\item $W_2^2(\mu_k,\gamma)=2sk^{2-t}+O(k^{1-t}\left(\log k\right)^{\frac{1}{2}})$,
	\item $m_2(\mu_k)-m_2(\gamma)=2sk^{2-t}+\frac{s}{4}m_2(\gamma)k^{-t}+o(k^{-t})$,
	\item $sk^{p-t}+o(k^{p-t})\leq m_p(\mu_k)-m_p(\gamma)\leq 2^{2(p-1)}sk^{p-t}+o(k^{p-t})$ for any $p\in[1,\infty)$.
\end{enumerate}
\end{lemma}
In the proof of Lemma~\ref{lem:example}, we modify the weighted sum of Gaussian measures so as to remove the overlaps (see Figure~\ref{fig}). This facilitates the detailed computations and provides precise estimates for the Fisher information, the relative entropy, the distances and the moments. This leads to, in particular, instability for the Beckner--Hirschman inequality (Theorem~\ref{thm:ins-bhi-pw} and Theorem~\ref{thm:ins-bhi-ew}) and the observations on the sharp exponents given in Remark~\ref{rmk:lsi-sharp-L1}, Remark~\ref{rmk:lsi-sharp-w1}, and Remark~\ref{rmk:ti-sharp-w1}. The asymptotic estimates also reveal how such quantities are related to each other when the deficit converges to 0. We believe that these concrete estimates may be applied to other related inequalities.

After this paper has been announced in May 2018, another counterexamples were produced in~\cite{Eldan2019a}, where it was shown that the LSI is unstable in the Wasserstein distances and there is no dimension-free general stability for $W_2$. We note that the construction of the examples in~\cite{Eldan2019a} is in the same spirit as in this paper. They considered the mixture of two Gaussian measures and manipulated the barycenters, the weight, and the covariances to get the desired the deficit of the LSI and the Wasserstein distances. Campared to the example presented in this paper, it seems not easy to apply the counterexamples of~\cite{Eldan2019a} to the $L^p$ distances in the setting of the entropic uncertainty principle. Also, it seems not clear how the examples in~\cite{Eldan2019a} give similar arguments on the sharp exponents as in Remark~\ref{rmk:lsi-sharp-w1} and Remark~\ref{rmk:ti-sharp-w1}.

\subsection{Organization} 
The rest of the paper is organized as follows. In Section~\ref{S:BHI}, we provide basic facts about the Beckner--Hirschman inequality and discuss its relation to the sharp Hausdorff--Young inequality. We present the proof of Lemma~\ref{lem:example} in Section~\ref{S:example}. In Section~\ref{S:prfmain}, we prove the main results. Applying Lemma~\ref{lem:example}, we prove instability results for the log Sobolev inequality and Talagrand's transportation inequality. In Section~\ref{S:prf-bhi}, we prove the instability results for the Beckner--Hirschman inequality. 

\subsection{Notation} 
Let $a_k$ and $b_k$ be sequences of real numbers. We say $a_k=O(b_k)$ if there exist $k_0\in\N$ and $M>0$ such that $|a_k|\leq M|b_k|$ for all $k\geq k_0$. If $M$ depends on some parameters $p,q,\cdots$, then we use the notation $a_k=O_{p,q,\cdots}(b_k)$. We say $a_k=o(b_k)$ if for any $\varep>0$, there exists $k_0\in\N$ such that $|a_k|\leq \varep|b_k|$ for all $k\geq k_0$. For a set $A$ in $\R^n$, the indicator (or characteristic) function of $A$ is denoted by $\indi{A}$.

\section{The Beckner--Hirschman inequality}\label{S:BHI}
In this section, we discuss the Beckner--Hirschman inequality and its relation to the sharp Hausdorff--Young inequality. In particular, we review the stability result for the sharp Hausdorff--Young inequality by Christ~\cite{Christ2014a} and how it can be interpreted in terms of stability for the Beckner--Hirschman inequality heuristically. Together with the instability results (Theorem~\ref{thm:ins-bhi-pw} and Theorem~\ref{thm:ins-bhi-ew}), we can get a better idea what a possible stability result for the Bechner--Hirschman inequality would be. 

Let $h\in L^{2}(\R^{n})$ with $h\geq0$ and $\norm{h}_{2}=1$, then the Shannon entropy of $h$ is given by
\begin{align*}
	S(h)=-\int_{\R^{n}}h\log h\,dx.
\end{align*}
The Beckner--Hirschman inequality (the BHI for short) states that
\begin{align}\label{eq:BHineq}
	S(|h|^{2})+S(|\wh{h}|^{2})\geq n(1-\log 2),
\end{align}
where $\wh{h}(\xi) = \int_{\R^{n}}e^{-2\pi i x\cdot \xi}h(x)dx$. By differentiating the (non-sharp) Hausdorff--Young inequality in $p$ at $p=2$, Hirschman~\cite{Hirschman1957a} obtained $S(|h|^{2})+S(|\wh{h}|^{2})\geq 0$. He conjectured in~\cite{Hirschman1957a} that the Gaussian functions are extremal for the inequality and the best constant in the right hand side of~\eqref{eq:BHineq} is $n(1-\log 2)$.  Beckner~\cite{Beckner1975a} found the best constant in the Hausdorff--Young inequality for all $p\in[1,2]$, which gave an affirmative answer to the conjecture.
  
Even though the Gaussian functions satisfy the equality, it was an open problem to show that the Gaussians are the only optimizers. In~\cite{Lieb1990a}, Lieb characterized the classes of optimizers for the Hausdorff--Young inequality and the BHI. Indeed, he proved that every optimizer for a convolution operator with a Gaussian kernel is Gaussian. Equality holds in~\eqref{eq:BHineq} if and only if $h$ is of the form
\begin{align*}
	h(x)=ce^{-\brk{x,Jx}+x\cdot v},
\end{align*}
where $c\in\C$, $v\in \C^{n}$ and $J$ is a $n\times n$ real positive definite matrix (see~\cite{Carlen1991a}*{Remarks in p.207}).

Let $g(x)=2^{\frac{n}{4}}e^{-\pi|x|^{2}}$ and $dm=g(x)^{2}dx$. The Fourier--Wiener transform is defined by $\cW f=\frac{1}{g}(\wh{fg})$. Let $f\in L^{2}(dm)$ with $\norm{f}_{L^{2}(dm)}=1$. By the Plancherel theorem, we have $\norm{\cW f}_{L^{2}(dm)}=\norm{f}_{L^{2}(dm)}=1$. For a normalized function $f$ in $L^{2}(dm)$, we define the deficit of the LSI with respect to $dm$ by
\begin{align*}
	\delta_{c}(f)
	:= \frac{1}{2\pi}\int_{\R^{n}}|\nabla f|^{2}\,dm - \int_{\R^{n}}|f|^{2}\log |f|^{2}\,dm.
\end{align*}
We note that $\delta(f)=\delta_{c}(u_{f})$ where $u_{f}(x)=(f(2\sqrt{\pi}x))^{1/2}$. Applying the BHI~\eqref{eq:BHineq} with $h=fg$, Carlen~\cite{Carlen1991a} characterized the equality cases of the LSI by showing that
\begin{align}\label{eq:Carlen_deficit}
	\delta_{c}(f) - \int_{\R^{n}}|\cW f|^{2}\log |\cW f|^{2}dm
	= S(|fg|^{2})+S(|\wh{fg}|^{2})-n(1-\log 2)
	\geq 0.
\end{align}
We define the deficit of the BHI by $\delta_{\BH}(h)= S(|h|^{2})+S(|\wh{h}|^{2})-n(1-\log 2)$. Then it follows from~\eqref{eq:Carlen_deficit} that $\delta_{c}(f)\geq \delta_{\BH}(fg)$.

We review the stability result for the Hausdorff--Young inequality by Christ~\cite{Christ2014a} and investigate how it is related to stability for the BHI. Let $p\in[1,2]$, $q=p/(p-1)$, and $\bA_{p}=p^{1/2p}q^{-1/2q}$. For a complex-valued function $h\in L^{p}(\R^{n})$, the sharp Hausdorff--Young inequality by Babenko~\cite{Babenko1961a} and Beckner~\cite{Beckner1975a} states that 
\begin{align*}
	\norm{\wh{h}}_{q}\leq \bA_{p}^{n}\norm{h}_{p}.
\end{align*}
Lieb~\cite{Lieb1990a} showed that equality holds if and only if $h(x)=ce^{-Q(x)+x\cdot v}$ where $v\in \C^{n}$, $c\in\C$, and $Q$ is a positive definite real quadratic form. Let $\sG$ be the set of all optimizers for the Hausdorff--Young inequality. Define $\sP(\R^{n})$ to be the set of all polynomials $P:\R^{n}\to\C$ of the form $P(x)=-x\cdot Ax+b\cdot x+c$ where $b\in\C^{n}$, $c\in\C$, and $A$ is a symmetric, positive definite real matrix. Note that $\sG\setminus\{0\}=\{e^{P}:P\in\sP(\R^{n})\}$. Let $u\in \sG\setminus\{0\}$. The real tangent space to $\sG$ at $u$ is $T_{u}\sG=\{Pu:P\in\sP(\R^{n})\}$, and the normal space to $\sG$ at $u$ is
\begin{align}\label{eq:normalspace}
	N_{u}\sG=\left\{h\in L^{p}:\Re\left(\int_{\R^{n}}hP\ol{u}|u|^{p-2}dx\right)=0\right\}.
\end{align}
Let $\dist_{p}(h,\sG)=\inf_{u\in\sG}\norm{h-u}_{p}$. For each $p\in[1,2]$, there exists $\delta_{0}>0$ such that if a nonzero function $h$ satisfies $\dist_{p}(h,\sG)\leq \delta_{0}\norm{h}_{p}$, then $h$ can be written as $h=h^{\bot}+\pi(h)$ where $\pi(h)\in \sG$ and $h^{\bot}\in N_{\pi(h)}\sG$. Since $\norm{h^{\bot}}_{p}=\norm{h-\pi(h)}_{p}$ and $\pi(h)\in\sG$, we have $\norm{h^{\bot}}_{p}\geq \dist_{p}(h,\sG)$. For a function $h$ satisfying $\dist_{p}(h,\sG)\leq \delta_{0}\norm{h}_{p}$, we define $\dist^{\ast}_{p}(h,\sG)=\norm{h^{\bot}}_{p}$. The deficit of the Hausdorff--Young inequality is given by 
\begin{align*}
	\delta_{\HY}(h;p):=\bA_{p}^{n}-\frac{\norm{\wh{h}}_{q}}{\norm{h}_{p}}.
\end{align*}
Let $\bB_{p,n}=\frac{1}{2}(p-1)(2-p)\bA_{p}^{n}$. For $\eta>0$, we define 
\begin{align*}
	h^{\bot}_{\eta}=
	\begin{cases}
		h^{\bot}, 	& |h^{\bot}|\leq \eta |\pi(h)|,\\
		0,			& |h^{\bot}|> \eta |\pi(h)|.
	\end{cases}
\end{align*}

In~\cite{Christ2014a}, Christ proved the following quantitative Hausdorff--Young inequality. He showed a compactness result using combinatoric arguments, and then computed the second variation to obtain remainder terms for the Hausdorff--Young inequality. 

\begin{theorem}[{\cite{Christ2014a}*{Theorem 1.3}}]
	For each $n\geq 1$ and $p\in (1,2)$, there exist $\eta_{0},\gamma>0$ and $C,c>0$ such that for all $\eta\in(0,\eta_{0})$, if a nonzero function $h\in L^{p}(\R^{n})$ satisfies $\dist_{p}(h,\sG)\leq \eta^{\gamma}\norm{h}_{p}$, then $\delta_{\HY}(h;p)\geq R_{1}(h;p)+R_{2}(h;p)$ where
	\begin{align}\label{eq:R1R2}
		R_{1}(h;p)
		&= (\bB_{p,n}-C\eta)\norm{h}_{p}^{-p}\left(\int_{\R^{n}}|h^{\bot}_{\eta}|^{2}|\pi(h)|^{p-2}dx\right),\\
		R_{2}(h;p)
		&= c\eta^{2-p}\Paren{\frac{\dist_{p}(h,\sG)}{\norm{h}_{p}}}^{p-2}\Paren{\frac{\norm{h^{\bot}-h^{\bot}_{\eta}}_{p}}{\norm{h}_{p}}}^{2}.\nonumber
	\end{align}
\end{theorem}

By differentiating the sharp Hausdorff--Young inequality, one can derive the Beckner--Hirschman inequality. Indeed, let $h\in L^{1}(\R^{n})\cap L^{2}(\R^{n})$ with $\|h\|_{2}=1$. Since $\delta_{\HY}(h;p)\geq 0$ and $\delta_{\HY}(h;2)=0$, the derivatives of $\delta_{\HY}(h;p)$ with respect to $p$ at $p=2$ is less than or equal to 0, which yields
\begin{align*}
	-\frac{d}{dp}\delta_{\HY}(h,p)|_{p=2}
	=\frac{1}{4}\Paren{S(|h|^{2})+S(|\wh{h}|^{2})-n(1-\log 2)}\geq 0.
\end{align*}
A natural question is whether the same argument yields a stability result for the BHI from that of the Hausdorff--Young inequality. In what follows, we fix a function $h\in L^{1}(\R^{n})\cap L^{2}(\R^{n})$ that satisfies $\dist_{p}(h,\sG)\leq \delta_{0}\norm{h}_{p}$ and $\|h\|_{2}=1$ for all $p\in[1,2]$. Note that $h^{\bot}$ and $\pi(h)$ depend on $p$. We also assume the following:
\begin{enumerate}
	\item We can choose a constant $\delta_{0}$ to be uniform in $p\in[1,2]$.
	\item The constant $\eta$ in~\eqref{eq:R1R2} is independent of $p\in (1,2)$.
	\item We choose the constant $C=C(p)$ in~\eqref{eq:R1R2} such that $C$ is differentiable on $(1,2]$ and $C(2)=0$.
	\item $R_{1}(h;p)\geq 0$ for all $p\in(1,2)$.
	\item $h^{\bot}$ and $\pi(h)$ are differentiable with respect to $p$.
\end{enumerate}

We emphasize here that these assumptions are optimistic and speculative. Based on these assumptions, we have $\delta_{\HY}(h;p)\geq R_{1}(h;p)+R_{2}(h;p)\geq R_{1}(h;p)\geq 0$ and $\delta_{\HY}(h;2)=R_{1}(h;2)=0$. Taking the derivative with respect to $p$, we obtain
\begin{align*}
	S(|h|^{2})+S(|\wh{h}|^{2})-n(1-\log 2)
	= -4\frac{d}{dp}\Big(\bA_{p}^{n}-\frac{\norm{\wh{h}}_{q}}{\norm{h}_{p}}\Big)|_{p=2}
	\geq -4\frac{d}{dp}R_{1}(h;p)|_{p=2}
\end{align*}
and
\begin{align*}
	\frac{d}{dp}R_{1}(h;p)|_{p=2}
	&=  \frac{d}{dp}(\bB_{p,n}-C\eta)|_{p=2}\Big(\lim_{p\uparrow2} \int_{\R^{n}}|h^{\bot}_{\eta}|^{2}|\pi(h)|^{p-2}dx\Big)\\
	&=  -(\frac{1}{2}+C'(2)\eta)\Big(\lim_{p\uparrow2} \int_{\R^{n}}|h^{\bot}_{\eta}|^{2}|\pi(h)|^{p-2}dx\Big).
\end{align*}
Let $h$ be a nonnegative function and $L_{\eta}=\{x:|h^{\bot}(x)|\leq\eta|\pi(h)(x)|\}$, then $h^{\bot}_{\eta}=h^{\bot}\cdot \indi{L_{\eta}}$. By Fatou's lemma, we get
\begin{align*}
	\lim_{p\uparrow2} \int_{\R}|h^{\bot}_{\eta}|^{2}|\pi(h)|^{p-2}dx
	\geq \int_{\R}|h^{\bot}_{\eta}|^{2}dx 
	= \int_{L_{\eta}}|h-\pi(h)|^{2}dx. 
\end{align*}
Since $h-\pi(h)\in N_{\pi(h)}\sG$, it follows from~\eqref{eq:normalspace} that $\pi(h)$ is nonnegative with $\norm{\pi(h)}_{2}\leq 1$. Let
\begin{align*}
	\wt{\fG} =\{ u\in\sG: u\geq 0, \norm{u}_{2}\leq 1 \}.
\end{align*}
Note that the set of the optimizers for the BHI defined in~\eqref{eq:BH_G_{ab}}, $\fG$, is contained in $\wt{\fG}$ and $\pi(h)\in\wt{\fG}$. Let $\eta$ be small enough that $\frac{1}{2}+C'(2)\eta>0$, then we get 
\begin{align*}
	\delta_{\BH}(f) \geq C_{\eta}\dist_{2}(\wt{h},\wt{\fG})^{2},
\end{align*}
where $\dist_{2}(\wt{h},\wt{\fG})=\inf_{u\in \wt{\fG}}\|\wt{h}-u\|_{2}$ and
\begin{align*}
	\wt{h}(x)=
	\begin{cases}
		h(x), & x\in L_{\eta},\\
		\pi(h)(x), & x\notin L_{\eta}.
	\end{cases}
\end{align*} 
Our observation suggests that there could be a stability bound for the BHI in terms of $L^{2}$ or weaker distance than $L^{2}$ with respect to the Lebesgue measure. We remark that Theorem~\ref{thm:ins-bhi-pw} and Theorem~\ref{thm:ins-bhi-ew} do not contradict to this observation.

In Theorem~\ref{thm:ins-bhi-ew}, we show that the BHI is not stable in terms of $\dist_{L^{p}(dm_{\theta})}(\cdot,\fG)$ with normalization for $p>\theta>0$. In Remark~\ref{rmk:bhi-lp-conv}, we explain that our example constructed in Theorem~\ref{thm:ins-bhi-ew} does not give any instability results for the BHI when $\theta =0$. Note that $\dist_{2}(\cdot,\cdot)$ is the boundary case when $\theta=0$ and $p=2$.  Compared to Theorem~\ref{thm:ins-bhi-pw}, $\dist_{2}(\cdot,\cdot)$ can be seen as the case when $\lambda=0$ (so that $p\geq 2(\lambda+1)=2$). Furthermore, Theorem~\ref{thm:ins-bhi-pw} implies that $L^{2}$--stability would be best possible if exists.   
\section{Proof of main lemma}\label{S:example}
Before proving the main lemma, we give a simple observation.  
\begin{example}
Let $b\in\R^{n}$, $g_{b}(x)=e^{b\cdot x-\frac{|b|^{2}}{2}}$, and $d\nu_{b}=g_{b}d\gamma$. Since $g_{b}$ are the optimizers of the LSI, we have $\delta(g_{b})=0$ for all $b\in\R^{n}$. Indeed, a direct calculation yields that
\begin{align*}
	\I(\nu_{b})
	&= \int_{\R^{n}}\frac{|\nabla g_{b}|^{2}}{g_{b}}d\gamma
	= |b|^{2}\int_{\R^{n}}g_{b}d\gamma=|b|^{2},\\
	\H(\nu_{b})
	&= \int_{\R^{n}}g_{b}\log g_{b}d\gamma
	= \int_{\R^{n}}\Paren{b\cdot (x+b)-\frac{1}{2}|b|^{2}}d\gamma
	= \frac{1}{2}|b|^{2}	,\\
	m_{2}(\nu_{b})
	&= \int_{\R^{n}}|x|^{2}g_{b}d\gamma
	=\int_{\R^{n}}|x+b|^{2}d\gamma
	=n+|b|^{2}.
\end{align*}
We have $\I(\nu_{b})$, $\H(\nu_{b})$, and $m_{2}(\nu_{b})$ tend to $\infty$, as $|b|\to\infty$. We remark that the deficit does not see the behavior of the barycenter $b$, whereas the other quantities depend on $b$. Notice also that the measure $g_{b}d\gamma$ is not centered provided  $b\neq 0$. 
\end{example}

As discussed in the introduction, the idea of the proof is to consider the weighted sum of the optimizers $g_b$ for the LSI. To obtain the precise estimates for the relevant quantities, we cut the overlaps of the densities of the optimizers and connect them to get a $C^\infty$ density. 

\begin{figure}[t]
\centering
\includegraphics[scale=0.28]{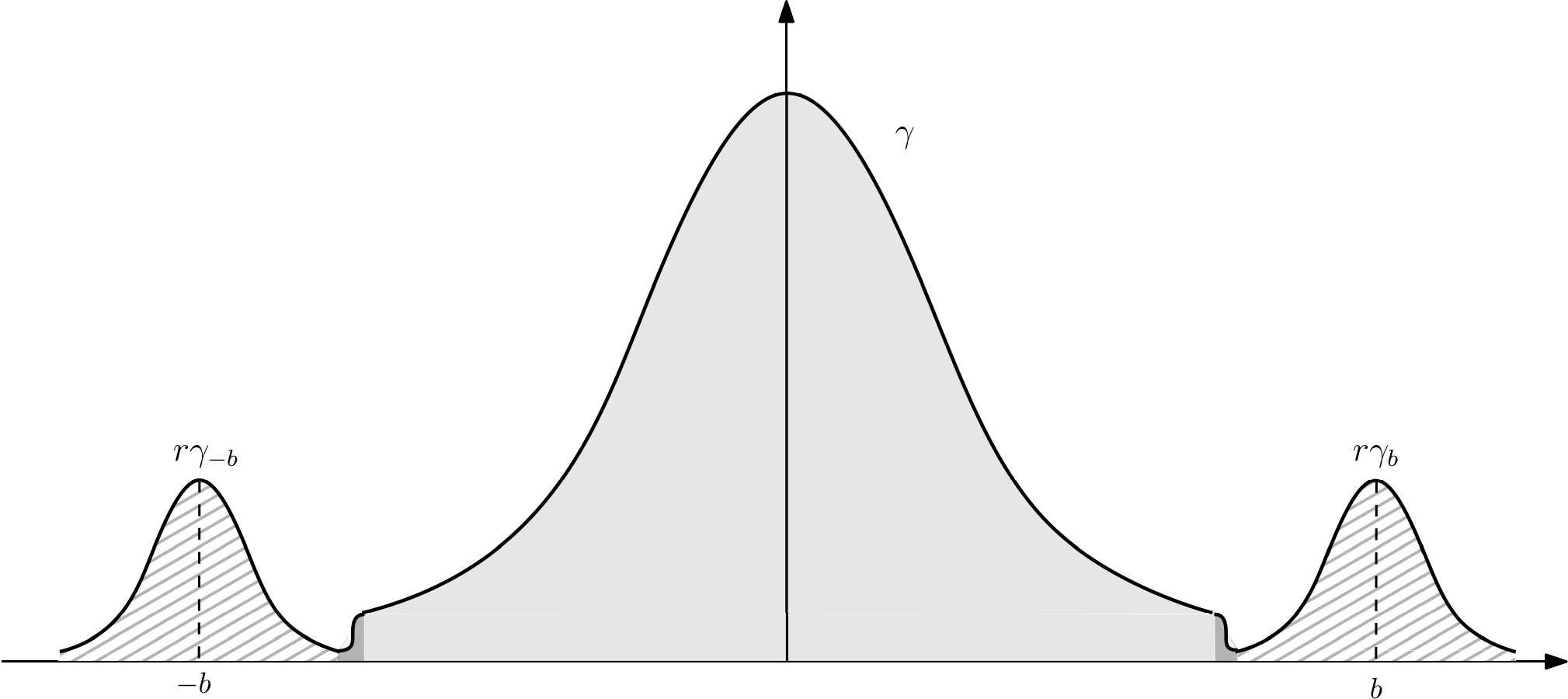}
\caption{The graph of $f_k(x)\gamma(x)$ constructed in the proof Lemma~\ref{lem:example}}
\label{fig}
\end{figure}

\begin{proof}[Proof of Lemma~\ref{lem:example}]
It suffices to consider the case $n=1$ for the following reason. Suppose that $\mu_{k}$ is the desired sequence of probability measures on $\R$. Let $\gamma_{n-1}$ the standard Gaussian measure on $\R^{n-1}$ and $\wt{\mu}_k=\mu_{k}\otimes \gamma_{n-1}$, then we have $\I(\wt{\mu}_k)=\I(\mu_{k})$, $\H(\wt{\mu}_k)=\H(\mu_{k})$, $\delta(\wt{\mu}_k)=\delta(\mu_{k})$, $m_{2}(\wt{\mu}_k)=(n-1)+m_{2}(\mu_{k})$, and $m_{p}(\wt{\mu}_k)\geq 2^{1-p} m_{p}(\mu_{k})-m_{p}(\gamma_{n-1})$. 
	
Let $d\gamma=(2\pi)^{-\frac{1}{2}}e^{-x^{2}/2}\,dx=\gamma(x)dx$ and $\Phi(x)=\int_{-\infty}^{x}d\gamma$. Let $s,t>0$ be fixed. We define a sequence of functions $\wt{f}_k$ in $C^{\infty}(\R)$ by
\begin{align*}
	\wt{f}_k(x)
	= \begin{cases}
        1, & |x|\le k-\frac{1}{2k}, \\
		L_{k}(x), & k-\frac{1}{2k}<|x|\le k, \\
        r_k e^{2k(x-k)}, & |x|>k,
	\end{cases}
\end{align*}
where $r_k=\frac{1}{4}\min\{sk^{-t},1\}$, and $L_{k}$ is a function in $C^{\infty}(\R)$ satisfying $L_k(k-\frac{1}{2k})=1$, $L_k(k)=r_k$, and
\begin{align*}
    r_k\le |L_k(x)|\le 1\quad \text{and} \quad |L_{k}'(x)|\leq 4k \quad \text{ for  } k-\frac{1}{2k}<|x|\le k.
\end{align*}
Note that $\wt{f}_k\in C^{\infty}(\R)\cap L^{1}(d\gamma)$.
Since $\Phi(k-\frac{1}{2k})-1,\Phi(k)-1$ are of order $e^{-ck^2}$ for some $c>0$ and $|L_k|\le 1$, we have
\begin{align*}
	\|\wt{f}_k\|_{L^{1}(d\gamma)}
	=2\left(\Phi(k-\frac{1}{2k})-\frac12+\int_{k-\frac{1}{2k}}^k L_k\,d\gamma + r_k\Phi(k)\right)
    =1+2r_k+O(e^{-ck^2}).
\end{align*}
Define $f_{k}=\wt{f}_k/\|\wt{f}_k\|_{L^{1}(d\gamma)}$ and $d\mu_{k}=f_{k}\,d\gamma$. Then we have
\begin{align*}
	\I(f_k)
	= \frac{2}{\|\wt{f}_k\|_{L^1(d\gamma)}}\left( \int_{k-\frac{1}{2k}}^{k}\frac{|L_{k}'(x)|^{2}}{L_{k}(x)}\,d\gamma +4r_k k^2 \Phi(k) \right)
\end{align*}
and
\begin{align*}
	\H(f_{k})
	&= \frac{1}{\|\wt{f}_k\|_{L^1(d\gamma)}}\left(\int \wt{f}_k\log \wt{f}_k\, d\gamma- \|\wt{f}_k\|_{L^1(d\gamma)}\log\|\wt{f}_k\|_{L^1(d\gamma)}\right)\\
     &= \frac{2}{\|\wt{f}_k\|_{L^1(d\gamma)}}\left( \int_{k-\frac{1}{2k}}^k L_k\log L_k\, d\gamma +\int_{-k}^\infty r_k\left(\log r_k+2k(x+k)\right)\,d\gamma \right)  -\log\|\wt{f}_k\|_{L^1(d\gamma)}\\
	&= \frac{2}{\|\wt{f}_k\|_{L^1(d\gamma)}}\left( \int_{k-\frac{1}{2k}}^k L_k\log L_k\, d\gamma+ 2k r_k \gamma(k)+r_k\left(2k^2 +\log r_k\right)\Phi(k) \right)  -\log\|\wt{f}_k\|_{L^1(d\gamma)}.
\end{align*}
Here, we used the fact that $\int_{-k}^\infty x\,d\gamma=\gamma(k)$.
It then follows that
\begin{align*}
	\delta(f_{k}) 
    &=\frac{2}{\|\wt{f}_k\|_{L^1(d\gamma)}}\left(\int_{k-\frac{1}{2k}}^{k} \left( \frac{|L_{k}'(x)|^{2}}{2L_{k}(x)} - L_k\log L_k \right) \, d\gamma-2k r_k \gamma(k) -r_k\log r_k\Phi(k) \right) \\
    &\qquad +\log\|\wt{f}_k\|_{L^1(d\gamma)}.
\end{align*}
Since $|L_k'(x)|^2/L_k(x)\le 16k^2/r_k$ and $r_k\le \frac14 sk^{-t}$ for all $k$, we have
\begin{align*}
	\int_{k-\frac{1}{2k}}^{k}\frac{|L_{k}'(x)|^{2}}{L_{k}(x)}\,d\gamma
    \leq \frac{8k}{r_k}\gamma(k-\frac{1}{2k}) = O(e^{-ck^2}).
\end{align*}
Since $|x\log x| \leq 1$ for all $x\in(0,1]$, one has
\begin{align*}
	\left|\int_{k-\frac{1}{2k}}^k L_k\log L_k\, d\gamma\right|
    \leq \frac{1}{2k}\gamma(k-\frac{1}{2k}) = O(e^{-ck^2}).
\end{align*}
Note that $\log\|\wt{f}_k\|_{L^1(d\gamma)}=2r_k+O(r_k^2)$. Therefore, we have 
\begin{align*}
	\delta(f_{k})
	&=-2r_k\log r_k+2r_k+O(r_k^2)+O(e^{-ck^2})\\
	&=-2r_k\log r_k+2r_k+o(r_k)\\
	&=-\frac{s}{2}k^{-t}\left(\log \frac{s}{4}-t\log k \right)
		+\frac{s}{2}k^{-t}
		+o(k^{-t})\\
	&=\frac{st}{2}k^{-t}\log k+\left(\frac{s}{2}\log \frac{4e}{s}\right)k^{-t}+o(k^{-t}),
\end{align*}
and
\begin{align*} 
	\H(f_k)
	&=4k^2 r_k +2r_k\log r_k-\log\|\wt{f}_k\|_{L^1(d\gamma)}+O(e^{-ck^2})\\
	&=s k^{2-t}+\frac{s}{2}k^{-t}\left(\log \frac{s}{4}-t\log k \right) - \frac{s}{2}k^{-t} +o(k^{-t})\\
	&=s k^{2-t}-\frac{st}{2}k^{-t}\log k-\left(\frac{s}{2}\log \frac{4e}{s}\right)k^{-t}+o(k^{-t}).
\end{align*}
Note that $\I(f_k)=4 k^2 r_k+O(e^{-ck^2}) =2sk^{2-t} +O(e^{-ck^2})$. 
By~\eqref{eq:HWI-cons}, we see
\begin{align*}
	2\H(f_k)-4\sqrt{\delta(f_k)\H(f_k)}
	\leq W_2^2(\mu_k,\gamma)
	\leq 2\H(f_k).
\end{align*}
Since
\begin{align*}
	\sqrt{\delta(f_k)\H(f_k)}
	&= \sqrt{\frac{s^2t}{2}k^{2(1-t)}\log k+o(k^{2(1-t)}\log k)}\\
	&=\sqrt{\frac{s^2t}{2}}k^{1-t}\left(\log k\right)^{\frac{1}{2}}+o(k^{1-t}\left(\log k\right)^{\frac{1}{2}}),
\end{align*}
we obtain 
\begin{align*}
    W_2^2(\mu_k,\gamma)
    =2s k^{2-t}+O(k^{1-t}\left(\log k\right)^{\frac{1}{2}})  
    =2s k^{2-t}+o(k^{2-t}).
\end{align*}
For the second moment, we see
\begin{align*}
	m_2(\mu_k)
	&=\frac{2}{\|\wt{f}_k\|_{L^1(d\gamma)}}\left( \int_0^{k-\frac{1}{2k}} |x|^2 \,d\gamma+\int_{k-\frac{1}{2k}}^k |x|^2 L_k\,d\gamma+r_k\int_k^\infty |x|^2 e^{2k (x-k)}\,d\gamma \right)\\
	&=1+\frac{2r_k}{\|\wt{f}_k\|_{L^1(d\gamma)}} \int_{-k}^\infty |x+2k|^2 \,d\gamma+O(e^{-ck^2})\\
	&=1+\frac{2r_k}{\|\wt{f}_k\|_{L^1(d\gamma)}}\left(m_2(\gamma)+4k\gamma(k)+4k^2\Phi(k) \right)+O(e^{-ck^2})\\
	&=1+2sk^{2-t}+o(k^{2-t}).
\end{align*}
Similarly, we have
\begin{align*}
	m_p(\mu_k)
    &= \frac{2}{\|\wt{f}_k\|_{L^1(d\gamma)}}\left( \int_0^{k-\frac{1}{2k}} x^p\, d\gamma +\int_{k-\frac{1}{2k}}^k  x^p L_k(x)\, d\gamma +r_k\int_k^\infty x^p e^{2k (x-k)}\, d\gamma \right)\\
	&= \frac{2}{\|\wt{f}_k\|_{L^1(d\gamma)}}\left( \frac{1}{2}m_p(\gamma) +r_k\int_{-k}^\infty (x+2k)^p \, d\gamma \right)+O(e^{-ck^2})\\
	&\geq m_p(\gamma)+2^{2}k^pr_k -2r_k m_p(\gamma)+O(e^{-ck^2})
\end{align*}
and
\begin{align*}
    m_p(\mu_k)\leq m_p(\gamma)+2^{2p}k^p r_k +2^p r_k m_p(\gamma)+O(e^{-ck^2}).
\end{align*}
\end{proof}

\section{Proofs of main results}\label{S:prfmain}
In this section, we present the proofs of instability results for the log Sobolev inequality and Talagrand's transportation inequality. The proofs are based on the construction of a sequence of probability measures and their asymptotic behaviors in Lemma~\ref{lem:example}. 

\subsection{Proof of Theorem~\ref{thm:ins-lsi-w2}}
Let $s=(M-n)/4$ and $t=2$. By Lemma~\ref{lem:example}, there exists a sequence of centered probability measures $\mu_k$ such that $\delta(\mu_k)\to 0$, $W_2^2(\mu_k,\gamma)=\frac{(M-n)}{2}+o(1)$, and $m_2(\mu_k)=n+\frac{(M-n)}{2}+o(1)$. Thus, we have $\mu_k\in \cP_2^M(\R^n)$ for large $k$ and 
\begin{align*}
	\lim_{k\to\infty}W_2^2(\mu_k,\gamma)=\frac{(M-n)}{2}>0.
\end{align*}
By H\"older's inequality and the fact that $z\log z \leq \frac{2}{p-1}|z-1|^{p}+2|z-1|$ for all $z\geq 0$, we have
\begin{align*}
	\H(\mu)
	\leq \frac{2}{p-1}\|f-1\|_{L^{p}(d\gamma)}^{p}+2\|f-1\|_{L^{p}(d\gamma)}
\end{align*}
for $d\mu=fd\gamma$. It then follows from $\H(\mu_k)=\frac{M-n}{4}+o(1)$ that
\begin{align*}
	\liminf_{k\to\infty}\|f_k-1\|_{L^p(d\gamma)}\geq C_{n,M,p}>0
\end{align*}
as desired.
\qed

\subsection{Proof of Theorem~\ref{thm:ins-lsi-w1}}
Since $W_1(\mu_k,\gamma)\leq W_p(\mu_k,\gamma)$ for all $p\geq 1$, it suffices to show the case $p=1$. Applying Lemma~\ref{lem:example} with $s=1$ and $t=\frac{1}{2}$, we get a sequence of centered probability measures $\mu_k$ such that $\delta(\mu_k)\to 0$ and $m_1(\mu_k)\to \infty$ as $k\to \infty$. Let $d\pi_k(x,y)$ be a coupling of $\mu_k$ and $\gamma$, then it follows from the triangle inequality that
\begin{align*}
	m_1(\mu_k)-m_1(\gamma)
	\leq \int |x-y|\,d\pi_k(x,y)
	\leq m_1(\mu_k)+m_1(\gamma).  
\end{align*}
Taking infimum over all couplings $\pi$, we get
\begin{align*}
	m_1(\mu_k)-m_1(\gamma)
	\leq W_1(\mu_k,\gamma)
	\leq m_1(\mu_k)+m_1(\gamma),
\end{align*}
which finishes the proof.
\qed

\subsection{Proof of Theorem~\ref{thm:ins-ti-w2}}
Let $s=(M-n)/4>0$ and $t=2$, then by Lemma~\ref{lem:example}, there exists a sequence of centered probability measures $\mu_k$ such that $\delta(\mu_k)\to 0$, $\H(\mu_k)=(M-n)/4+o(1)$, $W_2^2(\mu_k,\gamma)=\frac{(M-n)}{2}+o(1)$, and $m_2(\mu_k)=n+\frac{(M-n)}{2}+o(1)$. For large $k$, $\mu_k\in\cP_2^M(\R^n)$ and $\lim_{k\to\infty}W_2(\mu_k,\gamma)=(M-n)/2>0$. By~\eqref{eq:HWI-cons}, we have
\begin{align*}
	\delta_{\Tal}(\mu_k)^2\leq 16\H(\mu_k)\delta(\mu_k)\to 0,
\end{align*}
which completes the proof.
\qed

\subsection{Proof of Theorem~\ref{thm:ins-ti-w1}}
Let $\mu_k$ be the sequence of probability measures constructed in Lemma~\ref{lem:example} with $s=1$ and $t=\frac{p+1}{2}$, then $\mu_p(\mu_k)\to \infty$. Let $d\pi_k(x,y)$ be a coupling of $\mu_k$ and $\gamma$, then it follows from the inequality $|x+y|^p\leq 2^{p-1}(|x|+|y|)$ for $p\geq 1$ that
\begin{align*}
	2^{1-p}m_p(\mu_k)-m_p(\gamma)
	\leq \int |x-y|^p\,d\pi_k(x,y)
	\leq 2^{p-1}m_p(\mu_k)+2^{p-1}m_p(\gamma).  
\end{align*}
Taking infimum over all couplings $\pi$, we get
\begin{align*}
	2^{1-p}m_p(\mu_k)-m_p(\gamma)
	\leq W_p^p(\mu_k,\gamma)
	\leq 2^{p-1}m_p(\mu_k)+2^{p-1}m_p(\gamma)
\end{align*}
and $\lim_{k\to\infty}W_p(\mu_k,\gamma)=\infty$. It follows from~\eqref{eq:HWI-cons} and Lemma~\ref{lem:example} that
\begin{align*}
	\delta_{\Tal}(\mu_k)^2
	&\leq 16\H(\mu_k)\delta(\mu_k)\\
	&= 8s^2tk^{2(1-t)}\log k+o(k^{2(1-t)}\log k)\\
	&= 4(p+1)k^{1-p}\log k+o(k^{1-p}\log k)
\end{align*}
as desired.
\qed

\section{Proofs of instability for the Bechner--Hirschman inequality}\label{S:prf-bhi}

\subsection{Auxiliary lemmas}
To complete the proof of Theorem~\ref{thm:ins-bhi-ew}, we construct a sequence of functions $h_k\in L^p(dm_{\theta})$ from Lemma~\ref{lem:example}, and show that there exists a constant $C>0$ such that
\begin{align*}
	\dist_{L^{p}(dm_{\theta})}(h_{k},\fG)\geq C\norm{h_{k}}_{L^{p}(dm_{\theta})}
\end{align*} 
for large $k$. Lemma~\ref{BH_lem_4} and Lemma~\ref{BH_lem_1} provide estimates of the $L^p$ distance on the left hand side. To control the right hand side, we obtain a two-sided estimate of $\norm{h_{k}}_{L^{p}(dm_{\theta})}$ in Lemma~\ref{BH_lem_2}.

\begin{lemma}\label{BH_lem_4}
Let $p,\theta,a_0,w>0$ be such that $p>\theta>0$, $a_{0}>\pi$, and $0<w<(a_{0}/\pi)^{\frac{1}{4}}$. Let $G_{a}(x):=G_{a,0}(x)=(\frac{2a}{\pi})^{\frac{1}{4}}e^{-ax^{2}}$ and $M(a,w):=\{x:G_{a}(x)\geq wG_{\pi}(x)\}$. Then, there exist constants $C(p,a_{0},w), C(p,\theta)>0$ such that
\begin{align*}
	C(p,a_{0},w)a^{\frac{p-2}{4p}}
	\leq \norm{G_{a}\cdot \indi{M(a,w)}}_{L^{p}(dm_{\theta})}
	\leq C(p,\theta) a^{\frac{p-2}{4p}}
\end{align*}
for all $a\geq a_{0}$. In particular, if $p>2$ then $\lim_{a\to\infty}\norm{G_{a}\cdot \indi{M(a,w)}}_{L^{p}(dm_{\theta})}=\infty$.
\end{lemma}
\begin{proof}
Since $G_{a}$ is symmetric and decreasing in $[0,\infty)$, the level set $M_{a,w}=[-x_{0},x_{0}]$ where $x_{0}>0$ satisfies $G_{a}(x_{0})=wG_{\pi}(x_{0})$. Solving the equation for $x_{0}$, we obtain
\begin{align*}
	x_{0}=\frac{1}{2}\sqrt{\frac{\log a-\log \pi-4\log w}{a-\pi}}.
\end{align*}
Let $\beta=ap-\theta\pi>0$, then
\begin{align*}
	\norm{G_{a}\cdot \indi{M(a,w)}}_{L^{p}(dm_{\theta})}^{p}
	&= \int_{-x_{0}}^{x_{0}}|G_{a}(x)|^{p}\,dm_{\theta}\\
	&= \Big(\frac{2a}{\pi}\Big)^{\frac{p}{4}}
		\int_{-x_{0}}^{x_{0}}e^{-\beta x^{2}}\,dx\\
	&= 2^{\frac{p}{4}}\pi^{-\frac{p-2}{4}}a^{\frac{p-2}{4}}\left(p-\frac{\theta\pi}{a}\right)^{-\frac{1}{2}}
		(2\Phi(\sqrt{2\beta}x_{0})-1).
\end{align*}
Since $\sqrt{2\beta}x_{0}\to \infty$ as $a\to \infty$, there exists a constant $C(a_0,t)>0$ such that  $C(a_0,t)\leq 2\Phi(\sqrt{2\beta}x_{0})-1\leq 1$. We have
\begin{align*}
	2^{\frac{1}{4}}\pi^{-\frac{p-2}{4p}}p^{-\frac{1}{2p}} C(a_0,w)^{\frac{1}{p}} a^{\frac{p-2}{4p}}
	\leq \norm{G_{a}\cdot \indi{M(a,w)}}_{L^{p}(dm_{\theta})}
	\leq 2^{\frac{1}{4}}\pi^{-\frac{p-2}{4p}}(p-\theta)^{-\frac{1}{2p}} a^{\frac{p-2}{4p}},
\end{align*}
which completes the proof.
\end{proof}

Let $f_{k}$ be the sequence of functions defined in Lemma~\ref{lem:example} with $s=1$ and $t=\frac{1}{2}$. Note that the sequence $\{f_k\}$ is construted in Section~\ref{S:example} as follows: let $f_k=c_k \wt{f}_k\in C^\infty(\R)$ where 
\begin{align*}
	\wt{f}_k(x)
	= \begin{cases}
        1, & |x|\le k-\frac{1}{2k}, \\
		L_{k}(x), & k-\frac{1}{2k}<|x|\le k, \\
        r_k e^{2k(x-k)}, & |x|>k,
	\end{cases}
\end{align*}
where $r_k=\frac{1}{4\sqrt{k}}$, $c_k=\|\wt{f}_k\|_{L^1(d\gamma)}^{-1}$, and $L_{k}$ is a function in $C^{\infty}(\R)$ satisfying $L_k(k-\frac{1}{2k})=1$, $L_k(k)=\frac{1}{4\sqrt{k}}$, and
\begin{align*}
    \frac{1}{4\sqrt{k}}\le |L_k(x)|\le 1\quad \text{and} \quad |L_{k}'(x)|\leq 4k \quad \text{ for  } k-\frac{1}{2k}<|x|\le k.
\end{align*}
Define $h_{k}(x)=\sqrt{f_{k}(2\sqrt{\pi}x)}g(x)$, then it is easy to see $\norm{h_{k}}_{2}=\norm{f_{k}}_{L^{1}(d\gamma)}=1$.

\begin{lemma}\label{BH_lem_1}
Let $p>2$, $p>\theta>0$, and $h_{k}$ be defined as above. There exist $k_{0}\in\N$ and $a_{0}>\pi$ such that
\begin{align*}
	\norm{h_{k}-G_{a}}_{L^{p}(dm_{\theta})}\geq \norm{h_{k}-G_{\pi}}_{L^{p}(dm_{\theta})}
\end{align*}
for all $a\geq a_{0}$ and  $k\geq k_{0}$.
\end{lemma}
\begin{proof}
Let $\wt{G}_{a}(x)=G_{a}(\frac{x}{2\sqrt{\pi}})/G_{\pi}(\frac{x}{2\sqrt{\pi}})$ then
\begin{align*}
	\|h_{k}-G_{a}\|_{L^{p}(dm_\theta)}^{p}
	=(4\pi)^{\frac{\beta-1}{2}} \int |\sqrt{f_{k}(x)}-\wt{G}_{a}(x)|^{p}\gamma^{\beta}(x)\,dx,
\end{align*}
where $\gamma(x)=(2\pi)^{-\frac{1}{2}}e^{-\frac{|x|^{2}}{2}}$ and $\beta=\frac{p-\theta}{2}$. We choose $k_{0}\in\N$ such that $\frac{1}{2}\leq c_{k}\leq \frac{3}{2}$ for all $k\geq k_{0}$. Since $L_{k}(x)\leq 1$, we have $|\sqrt{c_{k}L_{k}(x)}-1|\leq 1$. Let $k\geq k_{0}$, then we get
\begin{align}\label{eq:c1ptheata}
	\int |\sqrt{f_{k}(x)}-1|^{p}\gamma^{\beta}(x)dx 
	&\leq 
	C_{1}(p,\theta)+2\int_{k}^{\infty}|\sqrt{f_{k}(x)}-1|^{p}\gamma^{\beta}(x)dx.  
\end{align}
Choose $a_{1}>\pi$ so that $\wt{G}_{a}(1)\leq \frac{1}{2}\leq \sqrt{c_{k}}$ for all $a\geq a_{1}$. 
Setting $A=\{x\in[-(k-\frac{1}{2k}),k-\frac{1}{2k}]:\wt{G}_{a}(x)\geq\frac{3}{2}\}$, then
\begin{align*}
	\int_{-k+\frac{1}{2k}}^{k-\frac{1}{2k}}|\sqrt{c_{k}}-\wt{G}_{a}(x)|^{p}\gamma^{\beta}(x)dx
	&\geq  \int_{A}\Big|\wt{G}_{a}(x)-\frac{3}{2}\Big|^{p}\gamma^{\beta}(x)dx \\
	&\geq  2^{1-p}\int_{A}|\wt{G}_{a}(x)|^{p}\gamma^{\beta}(x)dx-\Big(\frac{3}{2}\Big)^{p}(2\pi)^{-\frac{\beta-1}{2}}\beta^{-\frac{1}{2}}
\end{align*} 
for all $a\geq a_{1}$. Let $B=\{x\geq k:\sqrt{f_{k}(x)}\geq 1\}$. Since $\wt{G}_{a}(x)\leq 1$ for $x\geq k$, we have
\begin{align*}
	\int_{k}^{\infty}|\sqrt{f_{k}(x)}-\wt{G}_{a}(x)|^{p}\gamma^{\beta}(x)\,dx
	&\geq  \int_{B}|\sqrt{f_{k}(x)}-1|^{p}\gamma^{\beta}(x)\,dx \\
	&\geq  \int_{k}^{\infty}|\sqrt{f_{k}(x)}-1|^{p}\gamma^{\beta}(x)\,dx -\frac{1}{2}(2\pi)^{-\frac{\beta-1}{2}}\beta^{-\frac{1}{2}}
\end{align*}
and
\begin{align*}
	&\int |\sqrt{f_{k}(x)}-\wt{G}_{a}(x)|^{p}\gamma^{\beta}(x)dx\\
	&\qquad\geq 
	\int_{-k+\frac{1}{2k}}^{k-\frac{1}{2k}}|\sqrt{c_{k}}-\wt{G}_{a}(x)|^{p}\gamma^{\beta}(x)dx
	+2\int_{k}^{\infty}|\sqrt{f_{k}(x)}-\wt{G}_{a}(x)|^{p}\gamma^{\beta}(x)dx
	\\
	&\qquad\geq 
	2^{1-p}\int_{A}|\wt{G}_{a}(x)|^{p}\gamma^{\beta}(x)dx
	+2\int_{k}^{\infty}|\sqrt{f_{k}(x)}-1|^{p}\gamma^{\beta}(x)dx
	-C_{2}(p,\theta).
\end{align*}
By Lemma~\ref{BH_lem_4}, one can choose $a_{0}\geq a_{1}$ such that 
\begin{align*}
	\int_{A}|\wt{G}_{a}(x)|^{p}\gamma^{\beta}(x)dx
	&\geq 2^{p-1}(C_{1}(p,\theta)+C_{2}(p,\theta))	
\end{align*}
for all $a\geq a_{0}$. By~\eqref{eq:c1ptheata}, we have
\begin{align*}
	\int |\sqrt{f_{k}(x)}-\wt{G}_{a}(x)|^{p}\gamma^{\beta}(x)dx
	&\geq 
	2^{1-p}\int_{A}|\wt{G}_{a}(x)|^{p}\gamma^{\beta}(x)dx+\int |\sqrt{f_{k}(x)}-1|^{p}\gamma^{\beta}(x)dx\\
	&\qquad
	-C_{1}(p,\theta)-C_{2}(p,\theta)\\
	&\geq\int |\sqrt{f_{k}(x)}-1|^{p}\gamma^{\beta}(x)dx,
\end{align*}
which finishes the proof.
\end{proof}

\begin{lemma}\label{BH_lem_2}
Let $p>\theta>0$ and $h_{k}$ be defined as above, then
\begin{align*}
	\norm{h_{k}}_{L^{p}(dm_{\theta})} =O_{p,\theta}\left(
	k^{-\frac{3}{4}}\exp\left(\frac{\theta k^{2}}{p-\theta}\right)  \right).
\end{align*}
\end{lemma}
\begin{proof}
Let $\beta=\frac{p-\theta}{2}$. A direct computation yields that
\begin{align*}
	\norm{h_{k}}_{L^{p}(dm_{\theta})}^{p}
	&= (4\pi)^{\frac{\beta-1}{2}} \int |f_{k}(x)|^{\frac{p}{2}}\gamma^{\beta}(x)dx\\
	&= |c_{k}|^{\frac{p}{2}}2^{\frac{\beta-1}{2}}\beta^{-\frac{1}{2}}(2\Phi(\sqrt{\beta}(k-\frac{1}{2k}))-1)
		+2|c_{k}|^{\frac{p}{2}}\int_{k-\frac{1}{2k}}^{k}|L_{k}(x)|^{\frac{p}{2}}\gamma^{\beta}(x)dx \\
	&\qquad+2^{\frac{\beta+1}{2}}|c_{k}r_k|^{\frac{p}{2}}\beta^{-\frac{1}{2}}e^{\frac{p\theta k^{2}}{p-\theta}}\Phi(\frac{pk}{\sqrt{\beta}}-\sqrt{\beta}k).
\end{align*}
Choose $k_{1}\in\N$ such that $c_{k}\in[\frac{1}{2},\frac{3}{2}]$ and $\Phi(\frac{pb}{2\sqrt{\beta}}-\sqrt{\beta}k)\geq\frac{1}{2}$ for all $k\geq k_{1}$. Then we have
\begin{align*}
	\norm{h_{k}}_{L^{p}(dm_{\theta})}
	\geq C(p,\theta)k^{-\frac{3}{4}}e^{\frac{\theta k^{2}}{p-\theta}}.
\end{align*}
Since we have
\begin{align*}
	|c_{k}|^{\frac{p}{2}}2^{\frac{\beta-1}{2}}\beta^{-\frac{1}{2}}(2\Phi(\sqrt{\beta}(k-\frac{1}{2k}))-1)
		+2|c_{k}|^{\frac{p}{2}}\int_{k-\frac{1}{2k}}^{k}|L_{k}(x)|^{\frac{p}{2}}\gamma^{\beta}(x)dx\leq C(p,\theta),
\end{align*}
we can choose $k_{2}\in\N$ such that
\begin{align*}
	\norm{h_{k}}_{L^{p}(dm_{\theta})}
	\leq  C(p,\theta)k^{-\frac{3}{4}}e^{\frac{\theta k^{2}}{p-\theta}}
\end{align*}
for all $k\geq k_{2}$. 
\end{proof}

\subsection{Proof of Theorem~\ref{thm:ins-bhi-ew}}
Let $f_{k}$ be the sequence of functions constructed in the proof of Lemma~\ref{lem:example} with $s=1$ and $t=\frac{1}{2}$. Define $h_{k}(x)=\sqrt{f_{k}(2\sqrt{\pi}x)}g(x)$. Note that $\norm{h_{k}}_{L^{2}(dm_{\theta})}=\norm{f_{k}}_{L^{1}(d\gamma)}=1$.
Note also that $\delta(f_k)=\delta_c((f_k(2\sqrt{\pi}x)^{1/2})\ge \delta_{\BH}(h_k)$ by~\eqref{eq:Carlen_deficit}. 
Thus, it follows from Lemma~\ref{lem:example} that $\delta_{\BH}(h_{k})\to 0$ as $k\to\infty$. Since the function $h_{k}$ and $g^{-\theta}$ are symmetric and the symmetric rearrangement of $G_{a,r}$ is $G_{a}$, it follows from the rearrangement inequality (see~\cite{Lieb2001a}*{Theorem 3.5}) that
\begin{align*}
	\dist_{L^{p}(dm_{\theta})}(h_{k},\fG)
 	= \inf_{a\in(\frac{\theta\pi}{p},\infty)}\norm{h_{k}-G_{a}}_{L^{p}(dm_{\theta})} 
\end{align*}
for all $k\geq 1$. Here, we used the fact that 
\begin{align}\label{eq:Lpdmcondition}
	G_{a,r}\in L^{p}(dm_{\theta}) \text{ if and only if } a>\theta \pi / p.
\end{align}
Our goal is to show that there exists a constant $C=C(p,\theta)>0$ such that 
\begin{align*}
	\norm{h_{k}-G_{a}}_{L^{p}(dm_\theta)} \geq C\norm{h_{k}}_{L^{p}(dm_\theta)}
\end{align*}
for all $a\in(\frac{\theta\pi}{p},\infty)$ and for large $k$.

\subsubsection*{Case 1: $a\geq\pi$}
If $p>2$, it follows from Lemma~\ref{BH_lem_1} that there exists $a_{0}>\pi$ such that
\begin{align*}
	\dist_{L^{p}(dm_{\theta})}(h_{k},\fG)
	=\inf_{a\in(\frac{\theta\pi}{p},a_{0}]}\norm{h_{k}-G_{a}}_{L^{p}(dm_\theta)} 
\end{align*}
for large $k$. Thus, it suffices to show that if $k$ is large enough, then $\norm{h_{k}-G_{a}}_{L^{p}(dm_\theta)}\geq C\norm{h_{k}}_{L^{p}(dm_\theta)}$ for all $a\in(\pi,a_{0}]$. Since
\begin{align}\label{eq:LpGa}
	\norm{G_{a}}_{L^{p}(dm_\theta)}^{p}
	&=2^{\frac{p-\theta}{4}}\left(\frac{a}{\pi}\right)^{\frac{p-2}{4}}\left(p-\frac{\theta\pi}{a}\right)^{-\frac{1}{2}}\\
	&=C(p,\theta)a^{\frac{p-2}{4}}\left(p-\frac{\theta\pi}{a}\right)^{-\frac{1}{2}}\nonumber
\end{align}
is uniformly bounded in $a\in[ \pi,a_{0}]$,  we can choose $k_{1}\in\N$ so that for all $k\geq k_{1}$, $\norm{h_{k}}_{L^{p}(dm_\theta)}\geq 2\sup_{a\in[ \pi,a_{0}]}\norm{G_{a}}_{L^{p}(dm_\theta)}$ by Lemma~\ref{BH_lem_2}. We obtain
\begin{align*}
	\norm{h_{k}-G_{a}}_{L^{p}(dm_\theta)}
	&\geq  \norm{h_{k}}_{L^{p}(dm_\theta)}-\sup_{a\in[ \pi,a_{0}]}\norm{G_{a}}_{L^{p}(dm_\theta)} \\
	&\geq  \frac{1}{2}\norm{h_{k}}_{L^{p}(dm_\theta)}
\end{align*}
for all $a \in[\pi, a_{0}]$ and $k\geq k_{1}$. 

If $p\leq 2$, then it follows from~\eqref{eq:LpGa} that $\norm{G_{a}}_{L^{p}(dm_\theta)}^{p}
\leq C(p,\theta)\pi^{\frac{p-2}{4}}(p-\theta)^{-\frac{1}{2}}$ for all $a\geq \pi$.
By Lemma~\ref{BH_lem_2}, we choose $k_{2}\in\N$ such that $\norm{h_{k}-G_{a}}_{L^{p}(dm_\theta)}\geq \frac{1}{2}\norm{h_{k}}_{L^{p}(dm_\theta)}$ for all $k\geq k_{2}$.

\subsubsection*{Case 2: $\frac{\theta\pi}{p}<a<\pi$}
By Lemma~\ref{BH_lem_2}, it suffices to show that there exists a constant $c>0$ such that  
\begin{align*}
	\norm{h_{k}-G_{a}}_{L^{p}(dm_\theta)}
			\geq c k^{-\frac{3}{4}}e^{\frac{\theta k^{2}}{p-\theta}}
\end{align*}
for all $a\in (\frac{\theta\pi}{p},\pi)$ and large $k$. Let $\beta=\frac{p-\theta}{2}$ and $v=1-\frac{a}{\pi}$, then $0<v<1-\frac{\theta}{p}$. We define $R_{v,k}(x)=\wt{G}_{a}(x)/\sqrt{f_{k}(x)}$, then 
\begin{align*}
	\norm{h_{k}-G_{a}}_{L^{p}(dm_{\theta})}
	&=(4\pi)^{\frac{\beta-1}{2}}\int |\sqrt{f_{k}}-\wt{G}_{a}|^{p}\gamma^{\beta}dx \\
	&=(4\pi)^{\frac{\beta-1}{2}}\int |1-R_{v,k}|^{p}|f_{k}|^{\frac{p}{2}}\gamma^{\beta}dx\\
	&\geq (4\pi)^{\frac{\beta-1}{2}}|c_{k}r_k|^{\frac{p}{2}}
    \int_{k}^{\infty} |1-R_{v,k}|^{p}e^{pk(x-k)}\gamma^{\beta}dx.
\end{align*}
Let $Q_{v,k}(x)=\frac{v}{4}(x-\frac{2k}{v})^{2}-(\frac{1-v}{v})k^{2}$, then
\begin{align*}
	R_{v,k}(x)=\frac{(1-v)^{\frac{1}{4}}}{(c_{k}r_k)^{\frac{1}{2}}}e^{Q_{v,k}(x)}
\end{align*}
for $x\geq k$. Choose $w\in (1,\frac{p}{p-\theta})$, then
\begin{align*}
	Q_{v,k}(2wk)
	= v\left(wk-\frac{k}{v}\right)^{2}-\left(\frac{1-v}{v}\right)k^{2} 
	= w^{2}k^{2}\left(v-\left(\frac{2w-1}{w^{2}}\right)\right).
\end{align*}
Since the map $z\mapsto \frac{2z-1}{z^{2}}$ is decreasing on $(1,\frac{p}{p-\theta})$, we know 
\begin{align*}
	\frac{2w-1}{w^{2}}
	\geq \frac{2\left(\frac{p}{p-\theta}\right)-1}{\left(\frac{p}{p-\theta}\right)^{2}}
	=\frac{p^{2}-\theta^{2}}{p^{2}}
	>\frac{p-\theta}{p}.
\end{align*}
Since $v\in(0,\frac{p-\theta}{p})$, we have $Q_{v,k}(2wk)<0$. The function $Q_{v,k}(x)$ is symmetric about $x=\frac{2k}{v}$ and $\frac{2k}{v}>2wk$. This yields that $Q_{v,k}(x)\leq Q_{v,k}(wb)$ for all $x\in[2wk,\frac{4k}{v}-2wk]$. Thus we can choose $k_{3}\in\N$ so that $R_{v,k}(x)\leq \frac{1}{2}$ for all $k\geq k_{3}$ and $v\in(0,\frac{p-\theta}{p})$. Since $(w-\frac{p}{p-\theta})<0$ and $(\frac{2}{v}-w-\frac{p}{p-\theta})\geq c>0$ uniformly in $v$, we can choose $k_{4}\in\N$ so that 
\begin{align*}
	\Phi\left(2k\sqrt{\beta}\left(\frac{2}{v}-w-\frac{p}{p-\theta}\right)\right)-\Phi\left(2k\sqrt{\beta}\left(w-\frac{p}{p-\theta}\right)\right)\geq \frac{1}{2}
\end{align*}
for all $k\geq k_{4}$ and $v\in(0,\frac{p-\theta}{p})$. If $k$ is large enough, then we obtain
\begin{align*}
	\norm{h_{k}-G_{a}}_{L^{p}(dm_{\theta})}^{p}
	&\geq   (4\pi)^{\frac{\beta-1}{2}}2^{-p}|c_{k}r_k|^{\frac{p}{2}}
    \int_{k}^{\infty} e^{pk(x-k)}\gamma^{\beta}dx \\
	&\geq 2^{\frac{\beta-1}{2}-p}|c_{k}r_k|^{\frac{p}{2}}
		e^{\frac{p\theta k^{2}}{p-\theta}}\beta^{-\frac{1}{2}}
        \Bigg(\Phi\left(2k\sqrt{\beta}\left(\frac{2}{v}-w-\frac{p}{2\beta}\right)\right)\\
    &\qquad -\Phi\left(2k\sqrt{\beta}\left(w-\frac{p}{2\beta}\right)\right)\Bigg)\\
	&\geq C(p,\theta)k^{-\frac{3p}{4}}e^{\frac{p\theta k^{2}}{p-\theta}}.
\end{align*}
By Lemma~\ref{BH_lem_2}, we have 
\begin{align*}
	\norm{h_{k}-G_{a}}_{L^{p}(dm_\theta)}\geq C\norm{h_{k}}_{L^{p}(dm_\theta)}
\end{align*}
for all $a\in(\frac{\theta\pi}{p},\pi)$, which completes the proof.
\qed

\subsection{Proof of Theorem~\ref{thm:ins-bhi-pw}}

We note that $G_{a,r}\in L^{p}(d\eta_{\lambda})$ for all $a>0$ and $r\in\R$. Indeed we have
\begin{align}\label{eq:BH2Gar}
	\|G_{a,r}\|_{L^{p}(d\eta_{\lambda})}^{p}
	&=\int |G_{a,r}(x)|^{p}\,d\eta_{\lambda}\\
	&\leq \int |G_{a}(x)|^{p}\,d\eta_{\lambda}\nonumber \\
	&= \Big(\frac{2a}{\pi}\Big)^{\frac{p}{4}}\int |x|^{\lambda} e^{-apx^{2}}\,dx\nonumber\\
	&= \Big(\frac{2a}{\pi}\Big)^{\frac{p}{4}}(2ap)^{-\frac{\lambda+1}{2}}\int |x|^{\lambda} e^{-\frac{x^{2}}{2}}\,dx\nonumber\\
	&= C(p,\lambda)a^{\frac{p-2\lambda-2}{4}}m_{\lambda}(\gamma)\nonumber,
\end{align}
where $m_{\lambda}(\gamma)$ is the $\lambda$-th moment of the standard Gaussian measure. Let $f_{k}$ be the sequence of functions constructed in the proof of Lemma~\ref{lem:example} with $s=1$ and $t\in(\frac{2(p-\lambda)}{p},2)$. Define $h_{k}(x)=\sqrt{f_{k}(2\sqrt{\pi}x)}g(x)$, then
\begin{align*}
	\|h_{k}\|_{L^{p}(d\eta_{\lambda})}^{p}
	&=  C(p,\lambda)\int |f_{k}(x)|^{\frac{p}{2}}\gamma^{\frac{p}{2}}(x)|x|^{\lambda}dx\\
	&\geq C(p,\lambda)|c_{k}r_{k}|^{\frac{p}{2}}\int_{k}^{\infty} |x|^{\lambda}e^{-\frac{p}{4}(x-2k)^{2}}dx\\
	&= C(p,\lambda)|c_{k}r_{k}|^{\frac{p}{2}}\int_{-k}^{\infty} |x+2k|^{\lambda}e^{-\frac{p}{4}x^{2}}dx\\
	&\geq C(p,\lambda)|c_{k}r_{k}|^{\frac{p}{2}}(|2k|^{\lambda}-m_{\lambda}(\gamma))
\end{align*}
so that $\|h_{k}\|_{L^{p}(d\eta_{\lambda})}\to \infty$ as $k\to\infty$. By the rearrangement inequality,
\begin{align*}
	\dist_{L^{p}(d\eta_{\lambda})}(h_{k},\fG)=\inf_{a>0}\|h_{k}-G_{a}\|_{L^{p}(d\eta_{\lambda})}.
\end{align*}

Assume $p=2\lambda+2$, then $\|G_{a}\|_{L^{p}(d\eta_{\lambda})}=C(p,\lambda)m_{\lambda}(\gamma)$ is independent of $a$. We pick $k_{1}\in \N$ such that $\|h_{k}\|_{L^{p}(d\eta_{\lambda})}\geq 2 \|G_{a}\|_{L^{p}(d\eta_{\lambda})}$ for all $k\geq k_{1}$, then 
\begin{align*}
	\|h_{k}-G_{a}\|_{L^{p}(d\eta_{\lambda})}
	\geq \|h_{k}\|_{L^{p}(d\eta_{\lambda})}-\|G_{a}\|_{L^{p}(d\eta_{\lambda})}
	\geq \frac{1}{2}\|h_{k}\|_{L^{p}(d\eta_{\lambda})}
\end{align*}
for all $k\geq k_{1}$, as desired.

Suppose $p-2\lambda-2>0$. By~\eqref{eq:BH2Gar}, we have $\|G_{a}\|_{L^{p}(d\eta_{\lambda})}\to \infty$ as $a\to\infty$. Since $\|h_{k}\|_{L^{p}(d\eta_{\lambda})}\to \infty$ and $\|G_{a}\|_{L^{p}(d\eta_{\lambda})}$ is bounded in $a\in (0,a_{0}]$ for a fixed $a_{0}$ by~\eqref{eq:BH2Gar}, it suffices to show that there exist $k_{0}$ and $a_{0}$ such that
\begin{align*}
	\|h_{k}-G_{a}\|_{L^{p}(d\eta_{\lambda})}
	\geq \|h_{k}-G_{\pi}\|_{L^{p}(d\eta_{\lambda})}
\end{align*}
for all $k\geq k_{0}$ and $a\geq a_{0}$. Let $\wt{G}_{a}(x)=G_{a}(\frac{x}{2\sqrt{\pi}})/G_{\pi}(\frac{x}{2\sqrt{\pi}})$ then
\begin{align*}
	\|h_{k}-G_{a}\|_{L^{p}(d\eta_{\lambda})}^{p}
	=C(p,\lambda) \int |\sqrt{f_{k}(x)}-\wt{G}_{a}(x)|^{p}\gamma^{\frac{p}{2}}(x)|x|^{\lambda}dx.
\end{align*}
We choose $k_{1}\in\N$ such that $\frac{1}{2}\leq c_{k}\leq  \frac{3}{2}$ for all $k\geq k_{1}$. Let $I=[-x_{0},x_{0}]$ with 
\begin{align*}
	x_{0}=\frac{1}{2}\sqrt{\frac{\log a-\log \pi-4\log (3/2)}{a-\pi}},
\end{align*}
then $\wt{G}_{a}(x)\geq 3/2$ for all $x\in I$. Choose $a_{1}>\pi$ so that $\wt{G}_{a}(1)\leq \frac{1}{2}\leq \sqrt{c_{k}}$ for all $a\geq a_{1}$, then $I\subset [-(k-\frac{1}{2k}),k-\frac{1}{2k}]$. We get
\begin{align*}
	\int_{-(k-\frac{1}{2k})}^{k-\frac{1}{2k}} |\sqrt{f_{k}(x)}-\wt{G}_{a}(x)|^{p}\gamma^{\frac{p}{2}}(x)|x|^{\lambda}dx
	&\geq \int_{I} |\sqrt{f_{k}(x)}-\wt{G}_{a}(x)|^{p}\gamma^{\frac{p}{2}}(x)|x|^{\lambda}dx\\
	&\geq  C(p,\lambda)a^{\frac{p-2(\lambda+1)}{4}}\int_{I'}|x|^{\lambda}d\gamma-C_{1}(p,\lambda),
\end{align*}
where $I'=[-\sqrt{\frac{ap}{2\pi}}x_{0},\sqrt{\frac{ap}{2\pi}}x_{0}]$. Since $\sqrt{a}x_{0}\to \infty$ as $a\to \infty$, there exist $a_{2}$ and $C>0$ such that $\int_{I'}|x|^{\lambda}d\gamma\geq C$ for all $a\geq a_{2}$. Let $B=\{x:\sqrt{f_{k}(x)}\geq 1, |x|\geq k\}$. Since $\wt{G}_{a}(x)\leq 1$ for $x\geq k$, we have
\begin{align*}
	\int_{k}^{\infty}|\sqrt{f_{k}(x)}-\wt{G}_{a}(x)|^{p}\gamma^{\frac{p}{2}}(x)|x|^{\lambda}dx
	&\geq  \int_{B}|\sqrt{f_{k}(x)}-1|^{p}\gamma^{\frac{p}{2}}(x)|x|^{\lambda}dx \\
	&\geq  \int_{k}^{\infty}|\sqrt{f_{k}(x)}-1|^{p}\gamma^{\frac{p}{2}}(x)|x|^{\lambda}dx -C(p,\lambda).
\end{align*}
Combining our observation, we get
\begin{align*}
	&\int |\sqrt{f_{k}(x)}-\wt{G}_{a}(x)|^{p}\gamma^{\frac{p}{2}}(x)|x|^{\lambda}dx\\
	&\qquad\geq 
	\int_{-(k-\frac{1}{2k})}^{k-\frac{1}{2k}}|\sqrt{f_{k}(x)}-\wt{G}_{a}(x)|^{p}\gamma^{\frac{p}{2}}(x)|x|^{\lambda}dx
	+2\int_{k}^{\infty}|\sqrt{f_{k}(x)}-\wt{G}_{a}(x)|^{p}\gamma^{\frac{p}{2}}(x)|x|^{\lambda}dx
	\\
	&\qquad\geq
	C_{1}(p,\lambda)a^{\frac{p-2(\lambda+1)}{4}}
	+2\int_{k}^{\infty}|\sqrt{f_{k}(x)}-1|^{p}\gamma^{\frac{p}{2}}(x)|x|^{\lambda}dx -C_{2}(p,\lambda).
\end{align*}
We choose $k_{2}$ large enough so that for all $k\geq k_{2}$, we have
\begin{align*}
	\int |\sqrt{f_{k}(x)}-1|^{p}\gamma^{\frac{p}{2}}(x)|x|^{\lambda}dx 
	\leq C_{3}(p,\lambda)+2\int_{k}^{\infty}|\sqrt{f_{k}(x)}-1|^{p}\gamma^{\frac{p}{2}}(x)|x|^{\lambda}dx.
\end{align*}
It then follows that 
\begin{multline*}
	\int |\sqrt{f_{k}(x)}-\wt{G}_{a}(x)|^{p}\gamma^{\frac{p}{2}}(x)|x|^{\lambda}dx
	\geq C_{1}(p,\lambda)a^{\frac{p-2(\lambda+1)}{4}}
	+\int |\sqrt{f_{k}(x)}-1|^{p}\gamma^{\frac{p}{2}}(x)|x|^{\lambda}dx \\
	-C_{2}(p,\lambda)-C_{3}(p,\lambda).
\end{multline*}
Letting $a$ large enough, we obtain 
\begin{align*}
	\int|\sqrt{f_{k}(x)}-\wt{G}_{a}(x)|^{p}\gamma^{\frac{p}{2}}(x)|x|^{\lambda}dx
	\geq \int |\sqrt{f_{k}(x)}-1|^{p}\gamma^{\frac{p}{2}}(x)|x|^{\lambda}dx. 
\end{align*}
Therefore we have $\|h_{k}-G_{a}\|_{L^{p}(d\eta_{\lambda})}\geq \|h_{k}-G_{\pi}\|_{L^{p}(d\eta_{\lambda})}$ as desired.
\qed

\begin{remark}\label{rmk:bhi-lp-conv}
For the Lebesgue measure and $p\geq 0$, we have
\begin{align*}
	\|h_{k}-G_{\pi}\|^{p}_{p}
	&= (4\pi)^{\frac{p-2}{4}}\int |\sqrt{f_{k}}-1|^{p}\gamma^{\frac{p}{2}}(x)dx \\
    &= O(r_k^{\frac{p}{2}})+2(c_{k}r_k)^{\frac{p}{2}}\int_{k}^\infty\left|e^{k(x-k)}-1\right|^{p}\gamma^{\frac{p}{2}}(x)dx\\
	&= O(r_k^{\frac{p}{2}}).
\end{align*}
Thus, we get
\begin{align*}
	\lim_{k\to\infty}\dist_{L^{p}(dx)}(h_{k},\fG)
	\leq 	
	\lim_{k\to\infty}\|h_{k}-G_{\pi}\|_{p}=0,
\end{align*}
which implies that our example does not give an instability result for the BHI when $\theta=0$ in Theorem~\ref{thm:ins-bhi-ew} and $\lambda=0$ in Theorem~\ref{thm:ins-bhi-pw}.
\end{remark}

{\bf Acknowlegment.} The author would like to thank Prof.~Emanuel Indrei for suggesting this problem and his helpful advice while writing this paper, and Prof.~Rodrigo Ba\~nuelos for his invaluable help and encouragement. The author gratefully acknowledges the detailed comments and suggestions from anonymous referees, which greatly improved the content and the overall presentation.

\end{document}